\documentclass{amsart}

\usepackage{amsmath}%
\usepackage{amsfonts}%
\usepackage{amssymb}%
\usepackage{amsthm}
\usepackage{graphicx}
\usepackage{dsfont}
\usepackage{pdfcomment}
\usepackage{tikz}
\usepackage{tikz-3dplot}
\usepackage{subfigure}
\usepackage{soul}
\usepackage{enumerate}
\usepackage{cancel}

\usepackage{hyperref}
\usepackage{cleveref}
\usepackage{esvect}
\usepackage{bbm}

\usepackage{xcolor}

\newtheorem{theorem}{Theorem}[section]

\newtheorem{corollary}[theorem]{Corollary}

\newtheorem{definition}[theorem]{Definition}
\newtheorem{example}[theorem]{Example}

\newtheorem{lemma}[theorem]{Lemma}

\newtheorem{proposition}[theorem]{Proposition}
\newtheorem{remark}[theorem]{Remark}

\DeclareMathOperator{\tr}{tr}
\DeclareMathOperator{\adj}{adj}
\DeclareMathOperator{\Real}{Re}
\DeclareMathOperator{\Imaginary}{Im}
\renewcommand{\Re}{\Real}
\renewcommand{\Im}{\Imaginary}

\title[Period Matrices and Quasi-trees]{Period matrices and homological quasi-trees on discrete Riemann surfaces}

\author[W.Y. Lam]{Wai Yeung Lam}
\thanks{The first author was partially supported by the FNR grant CoSH O20/14766753.}
\address{Department of Mathematics, University of Luxembourg, Maison du nombre, 6 avenue de la Fonte, L-4364 Esch-sur-Alzette, Luxembourg}
\email{wyeunglam@gmail.com}
\author[O.-H.S. Lo]{On-Hei Solomon Lo}
\thanks{The second author was partially supported by the Fundamental Research Funds for the Central Universities.}
\address{School of Mathematical Sciences,
Key Laboratory of Intelligent Computing and Applications (Ministry of Education), Tongji University, 
Shanghai 200092, China}
\email{ohsolomon.lo@gmail.com}
\author[C.H. Yuen]{Chi Ho Yuen}
\thanks{The third author was partially supported by the Ministry of Science and Technology of Taiwan project MOST 114-2115-M-A49-015-MY3.}
\address{Department of Applied Mathematics, National Yang Ming Chiao Tung University, Hsinchu 30009, Taiwan}
\email{chyuen@math.nctu.edu.tw}

\begin{document}
	
	\maketitle

\begin{abstract}
We study discrete period matrices associated with graphs cellularly embedded on closed surfaces, resembling classical period matrices of Riemann surfaces. Defined via integrals of discrete harmonic $1$-forms, these period matrices are known to encode discrete conformal structure in the sense of circle patterns. We obtain a combinatorial interpretation of the discrete period matrix, where its minors are expressed as weighted sums over certain spanning subgraphs, which we call homological quasi-trees. Furthermore, we relate the period matrix to the determinant of the Laplacian for a flat complex line bundle. We derive a combinatorial analogue of the Weil–Petersson potential on the Teichmüller space, expressed as a weighted sum over homological quasi-trees. Finally, we study the collection of homological quasi-trees from a (delta-)matroidal perspective. The discrete period matrix plays a role similar to that of the response matrix in circular planar networks, thereby addressing a question posed by Richard Kenyon. 
\end{abstract}

\sloppy

\section{Introduction}\label{sec:intro}

A Riemann surface is a surface equipped with a conformal structure. Riemann surfaces arise in many areas of mathematics, including algebraic geometry, complex analysis, and mathematical physics. A key object associated with a Riemann surface is its {\em period matrix}, which encodes information about the conformal structure and defines the Jacobian variety. Classically, the period matrix is defined via line integrals of holomorphic $1$-forms over closed loops. In order to deduce a relation to combinatorics, we have to formulate the period matrix in terms of {\em harmonic $1$-forms}---the real parts of holomorphic $1$-forms.

Over a closed Riemann surface of genus $g$, the space of harmonic $1$-forms is of real dimension $2g$. Fixing a homology basis for $H_1(S;\mathbb{Z}) \cong \mathbb{Z}^{2g}$, each harmonic $1$-form is uniquely determined by its periods, which are the integrals of the harmonic $1$-form over the homology basis. For every harmonic $1$-form $\omega$, there is a conjugate harmonic $1$-form $\star\omega$ such that together they become the real and imaginary parts of a holomorphic $1$-form. Our period matrix is defined via the linear map sending the periods of a harmonic $1$-form to the periods of the conjugate harmonic $1$-form. It is a real $2g \times 2g$ invertible matrix.

Analogously, given a cellularly embedded graph equipped with positive edge weights, it is also associated with a period matrix by considering discrete harmonic $1$-forms. Let $(V,E,F)$ be a cell decomposition of a closed oriented surface $S$ of genus $g$, equipped with positive edge weights $\mathsf{c}: E \to \mathbb{R}_{>0}$. We denote by $\gamma_1, \gamma_2, \dots, \gamma_{2g}$ oriented simple loops on the surface that form a standard basis of the first homology $H_1(S;\mathbb{Z})$. Namely, the intersection matrix $\Omega$, with entries $\Omega_{i,j} = \iota(\gamma_i, \gamma_j)$ being the algebraic intersection number, is of the form
\[
\Omega = \begin{pmatrix}
	0 & \mathbb{I}_g \\
	-\mathbb{I}_g & 0
\end{pmatrix},
\]
where $\mathbb{I}_g$ is the $g \times g$ identity matrix. 

We now consider discrete harmonic $1$-forms and the space of periods. Throughout, the surface $S$ is endowed with a fixed orientation. The cell decomposition $(V,E,F)$ is viewed as an undirected embedded graph. We denote by $\vec E$ the set of oriented edges. Each face $\phi \in F$ is oriented consistently with the orientation of $S$, so that its boundary $\partial \phi$ is a cyclic sequence of oriented edges of the cell decomposition.

A {\em discrete $1$-form} is a function on oriented edges, $\omega: \vec{E} \to \mathbb{R}$, satisfying the antisymmetry condition $\omega_{uv} = -\omega_{vu}$ for every edge with its two opposite orientations $uv$ and $vu$. The $1$-form $\omega$ is said to be {\em closed} if, for every face $\phi \in F$, the sum over its boundary vanishes: $\sum_{uv \in \partial \phi} \omega_{uv} = 0$.

For a closed $1$-form, its {\em period} $A = (a_1, a_2, \dots, a_{2g})^t \in \mathbb{R}^{2g}$ is a column vector defined by
\[
a_k := \sum_{uv \in \gamma_k} \omega_{uv},
\]
where the sum is over oriented edges forming a cycle homotopic to $\gamma_k$. The closedness of $\omega$ implies that the period is independent of the chosen cycles. A closed $1$-form $\omega$ is said to be {\em harmonic} if its dual $1$-form $\star \omega := \mathsf{c} \omega$ is also closed on the dual cell decomposition; that is, for each $v \in V$,
\[
\sum_u \star\omega_{uv} = \sum_u \mathsf{c}_{uv} \omega_{uv} = 0,
\]
where $\mathsf{c}_{uv} = \mathsf{c}_{vu}$, and the sum is taken over all vertices $u$ adjacent to $v$. For an edge $uv$ oriented from $u$ to $v$, its dual edge is oriented from the right face of $uv$ to its left face. We denote by $A^{\star}$ the period of $\star \omega$, summed along cycles in the dual graph homologically equivalent to the homology basis. It is known that for any vector $A \in \mathbb{R}^{2g}$, there exists a unique discrete harmonic $1$-form $\omega$ with the prescribed period, and we denote $A^{\star}$ as the period of its dual $1$-form~\cite{Mercat2007,Bobenko2016}. This correspondence thus induces a linear action on the space of periods $\mathbb{R}^{2g}$. With respect to the homology basis, this linear map defines our $2g \times 2g$ period matrix $L$, satisfying
\[
A^{\star} = L A.
\]
The discrete period matrix is introduced by Mercat~\cite{Mercat2002, Mercat2007} and it is known that the matrix product $\Omega L$ is symmetric and positive definite, in accordance with the Riemann bilinear relation~\cite{Mercat2007,Bobenko2016}.

A cell-decomposed surface equipped with edge weights is regarded as a discrete Riemann surface~\cite{Mercat2001}; here the edge weight plays the role of conductance in an electric network that defines the graph Laplacian. We note that there are other ways to interpret a graph as a discrete analogue of Riemann surfaces~\cite{BakerNorine2007,Kontservich}, but we do not elaborate on them here for brevity. When the cell decomposition is a triangulation, the space of positive edge weights parametrizes the space of weighted Delaunay triangulations of Riemann surfaces with constant curvature~\cite{Lam2022,Lam2024harmonic}. The underlying circle pattern of a weighted Delaunay decomposition serves as a discrete conformal structure, generalizing William Thurston's circle packings~\cite{Stephenson2005,Glickenstein2024}. In this context, the discrete period matrix is related to the deformation space of circle patterns~\cite{Lam2021} and its Weil--Petersson symplectic form~\cite{Lam2024torus,Lam2024pull}. As the triangulations of the surface are refined, the discrete period matrices associated with geometric edge weights converge to the period matrix of the classical Riemann surface~\cite{Mercat2002,Bobenko2016,Ulrike2021}.

In this article, we shift the focus from geometric to combinatorial aspects of the discrete period matrix, allowing the edge weights to take arbitrary positive values. We show that the minors of the discrete period matrix can be expressed as weighted sums over \emph{homological quasi-trees} in the embedded $1$-skeleton graph $(V, E)$.

\begin{definition}\label{def:kquasitree}
    Given a topological embedding of a graph $f:(V,E) \to S$, we define a \emph{$k$-homological quasi-tree} $T$ to be a spanning connected subgraph with $|V|-1+k$ edges such that the induced homomorphism on homology $f_*: H_1(T;\mathbb{Z}) \to H_1(S;\mathbb{Z})$ is injective. 
    
\end{definition}

For brevity, a ($k$-)homological quasi-tree is simply called a ($k$-)quasi-tree for the rest of this article. Observe that $0$-quasi-trees correspond exactly to the spanning trees of the graph $(V,E)$. For a closed surface $S$ of genus $g$, the first homology group satisfies $H_1(S;\mathbb{Z}) \cong \mathbb{Z}^{2g}$. Therefore, it suffices to consider $k$-quasi-trees with $1 \leq k \leq 2g$. By definition, every cycle in a quasi-tree $T$ is non-null-homologous in $S$ and therefore neither contractible nor separating.

We associate intersection matrices to every embedded subgraph. To simplify notation, we denote $\gamma^1, \gamma^2, \dots, \gamma^{2g}$ a collection of oriented simple loops in $S$ satisfying the Kronecker delta relation $\iota(\gamma^i,\gamma_j) = \delta_{ij}$. For example, we may take \(\gamma^i=-\gamma_{i+g}\) for \(1\le i\le g\) and \(\gamma^i=\gamma_{i-g}\) for \(g+1\le i\le 2g\), but we will also use alternative choices arising from natural constructions (for instance, from the dual graph). Any such choice yields the same intersection matrices, since the intersection pairing depends only on the homology classes. 

For any connected subgraph $T$, we consider its cellular homology $H_1(T;\mathbb{Z})$. Let $\mu_1, \dots, \mu_k$ form a basis of $H_1(T;\mathbb{Z})$. Let $I \subset \{1, 2, \dots, 2g\}$ with $|I| = k$. Write $I = \{i_1, i_2, \dots, i_k\}$ with $i_1 < i_2 < \dots < i_k$, and define the $k \times k$ \emph{intersection matrix} $\mathcal{T}_{T,I}$ by
\[
(\mathcal{T}_{T,I})_{r,s} := \iota(\gamma^{i_r}, f_*\mu_s).
\]
We write $\mathcal{T}_{I}$ for $\mathcal{T}_{T,I}$ when the dependence on $T$ is clear from the context, and adopt the convention $\det \mathcal{T}_\emptyset = 1$.

For $I, J \subset \{1, 2, \dots, 2g\}$ with $|I| = |J| = k$, the quantity
\[
\det \mathcal{T}_{I} \cdot \det \mathcal{T}_{J}
\]
is a well-defined invariant of the subgraph $T$, independent of the choice of ordered basis of $H_1(T;\mathbb{Z})$.

For a $k$-quasi-tree $T$, its first homology $H_1(T;\mathbb{Z})$ has rank $k$ and it is associated with intersection matrices of size $k \times k$.

\begin{theorem} \label{thm:HQT_count}
Given a cell decomposition $(V,E,F)$ of a closed oriented surface $S$ of genus $g$, we have an embedding of the $1$-skeleton graph $f\colon (V,E) \to S$. Let $I, J \subset \{1,2,\dots,2g\}$ be subsets of size $k$, where $1 \leq k \leq 2g$. Then the determinant of the submatrix $(\Omega L)_{I,J}$ satisfies
\begin{align} \label{eq:Lminor}
\det(\Omega L)_{I,J} = 
\frac{
\sum_{\text{$k$-quasi-tree } T} \left( \det \mathcal{T}_{I} \cdot \det \mathcal{T}_{J} \cdot \prod_{e \in T} \mathsf{c}_e \right)
}{
\sum_{\text{spanning tree } T} \prod_{e \in T} \mathsf{c}_e 
}.
\end{align}
\end{theorem}

We note that while a $k$-quasi-tree may not contribute to the sum for every $I,J$, it always contributes to some pair of $I,J$.

This formula confirms that $\Omega L$ is positive definite and symmetric. Indeed, when $I=J$, we have $\det(\Omega L)_{I,J} >0$ and hence the principal minors have positive determinants. The symmetry of the matrix follows from the case $k=1$, where every $1$-quasi-tree $T$ has an oriented loop $\mu$ unique up to orientation:  
\begin{corollary}\label{cor:k1}
For every $i,j \in \{1,2,\dots,2g\}$, we have
\[
(\Omega L)_{i,j} = \frac{\sum_{\text{$1$-quasi-tree } T} \left(\iota(\gamma^i,f_* \mu) \cdot \iota(\gamma^j,f_* \mu) \cdot \prod_{e \in T} \mathsf{c}_{e}\right)}{\sum_{\text{spanning tree } T} \prod_{e \in T} \mathsf{c}_{e}} ,
\]
Here, $\mu$ is the unique (up to orientation) homology generator of the $1$-quasi-tree $T$.
\end{corollary}

Our theorem is closely related to Kirchhoff’s Matrix-Tree Theorem, which establishes a fundamental connection between combinatorial structures and determinantal formulas. The Matrix-Tree Theorem and its variations plays an important role in many areas that involve either of these subjects, including statistical mechanics and two-dimensional conformal field theory~\cite{Benjamin2001,Lawler2004}. These areas also provide part of the motivation for this work.

Given a graph $(V,E)$ equipped with edge weights, the \emph{discrete Laplacian} is the linear operator $\Delta \colon \mathbb{R}^V \to \mathbb{R}^V$ defined as follows. For $h \in \mathbb{R}^V$, the image $\Delta h \in \mathbb{R}^V$ is the function whose value at $v \in V$ is
\[
(\Delta h)_v := \sum_{uv} \mathsf{c}_{uv}\,(h_v - h_u),
\]
where the sum ranges over all edges $uv$ incident to $v$. Here, the graph is not necessarily realized on a surface, and multiple edges as well as loop edges are allowed. The discrete Laplacian acts on functions defined at the vertices and serves as a discrete analogue of the Laplace–Beltrami operator. 
The Matrix-Tree Theorem states that the product of the nonzero eigenvalues of $\Delta$, denoted $\det' \Delta$, equals the weighted sum over all spanning trees of the graph:
\[
\mbox{det}' (\Delta)= \sum_{\text{spanning tree } T}  \left(  \prod_{e \in T} \mathsf{c}_{e} \right),
\]
which appears in the denominator of Equation~\eqref{eq:Lminor}.

Particularly for planar graphs, there is a well-known bijection between spanning trees in the graph and spanning trees in the dual graph. This bijection extends naturally to homological quasi-trees on surfaces, and their homological invariants admit a natural correspondence (see Section~\ref{sec:duality}). For the rest of this article, given $E_0\subset E$, we denote by $E_0^\star:=\{e^\star:e\in E_0\}$ the set of dual edges corresponding to the primal edges in $E_0$. The matrix $\mathcal{T}^\star_{\overline{I}}$ appearing below is defined analogously to $\mathcal{T}_{I}$, whose precise definition is given in Section~\ref{sec:cellemb}.

\begin{theorem}\label{thm:bijection}
    For every integer $k \in [0,2g]$, the complement of a $k$-quasi-tree $T$ in the primal graph yields a $(2g-k)$-quasi-tree ${\overline{T}^{\star}}=\{e^\star:e\in E\setminus T\}$ in the dual graph. Furthermore, for any $I\subset \{1,2,\dots,2g\}$ with $|I|=k$, we have 
    \[
    \det \mathcal{T}_I = \pm\det \mathcal{T}^\star_{\overline{I}},
    \]
    where $\overline{I}:=\{1,2,\dots,2g\}\setminus I$. Upon fixing the ordering of $I,J,\overline{I},\overline{J}$, the multiplicative factor $\pm 1$ relating the two quantities $\det \mathcal{T}_I \det \mathcal{T}_J$ and $\det \mathcal{T}^\star_{\overline{I}}\det \mathcal{T}^\star_{\overline{J}}$ is independent of $T$.
\end{theorem}

It follows that $2g$-quasi-trees of the primal graph are in bijection with spanning trees of the dual graph. Combining this with Theorem~\ref{thm:HQT_count} yields an alternative expression for the determinant of $\Omega L$.

\begin{corollary}\label{cor:det}
    \begin{align*}
		\det(\Omega L)= \frac{\sum_{\text{dual spanning tree } {\overline{T}^{\star}}}  \left( \prod_{e \in {\overline{T}^{\star}}} \mathsf{c}^{-1}_{e} \right)}{\sum_{\text{spanning tree } T}  \left(  \prod_{e \in T} \mathsf{c}_{e} \right)} \cdot \prod_{e \in E} \mathsf{c}_{e},
    \end{align*}
    where the denominator is the sum over spanning trees $T$ of the primal graph, while the numerator is the sum over spanning trees ${\overline{T}^{\star}}$ of the dual graph.
\end{corollary}

Given an embedded graph on a surface, the counting of subgraphs has appeared in other settings. A setting closely related to ours is the Laplacian for a flat $\mathbb{C}$-bundle, where the holonomy along  $\gamma_1,\dots, \gamma_{2g}$ has multiplicative factors $z_1,z_2, \dots z_{2g} \in \mathbb{C}^*$. The bundle Laplacian $\Delta(z_1,\dots,z_{2g})$ is a $V \times V$ matrix depending on the variables $z_1,\dots,z_{2g}$. When $z_1=z_2=\dots=z_{2g}=1$, the bundle Laplacian becomes the usual discrete Laplacian. Associated to the bundle Laplacian, one considers a polynomial
\begin{align*}
	P(z_1,\dots,z_{2g}):= \det \Delta(z_1,\dots,z_{2g}).
\end{align*} 
The polynomial is a rich algebraic object and invariant under cluster mutations \cite{GK2013,KLRR2022}. It is interesting to determine how the polynomial $P$ encodes the discrete Riemann surface up to cluster mutations~\cite{Boutillier2023,George2024}. When the surface is a torus, i.e.\ $g=1$, the algebraic curve $\mathcal{C} := \{(z_1,z_2) \in (\mathbb{C}^2 \setminus \{0,0\}) : P(z_1,z_2) = 0\}$ is known to be a Harnack curve, with fruitful implications in statistical mechanics~\cite{KOS2006}. For arbitrary genus, the polynomial $P$ is known to count cycle-rooted spanning forests in the embedded graph~\cite{Forman1993,Kenyon2011}. At $z_1=z_2=\dots=z_{2g}=1$, one can readily verify that the polynomial $P$ and its first derivatives vanish. We denote by $ \operatorname{Hess}(P)$ the Hessian matrix of $P$, namely
\[
(\operatorname{Hess}(P))_{i,j} =  \frac{\partial^2 P}{\partial z_i \partial z_j}. 
\]

\begin{theorem}\label{thm:Hess}
	The Hessian of the polynomial $P(z_1, z_2, \dots, z_{2g})$ associated with the bundle Laplacian, evaluated at $z_1 = z_2 = \dots = z_{2g} = 1$, is related to the period matrix via
	\[
	\left. \operatorname{Hess}(P) \right|_{z_1 = \dots = z_{2g} = 1} = -2 \left( \sum_{\text{spanning tree } T} \prod_{e \in T} \mathsf{c}_e \right) \Omega L.
	\]
\end{theorem}

Theorem~\ref{thm:Hess} provides an alternative proof of Corollary~\ref{cor:k1}, based on the known expression relating the polynomial $P$ to cycle-rooted spanning forests in the graph (see Section~\ref{sec:crsf}).

Our results suggest an intriguing combinatorial interpretation of the potential function for the Weil--Petersson metric. In the classical setting, the Teichmüller space---namely, the moduli space of marked hyperbolic surfaces---is equipped with the Weil--Petersson Kähler metric \( h \). On one hand, via the spectral zeta function, the regularized determinant \( \log \det'\Delta \) of the Laplace--Beltrami operator defines a function on the Teichmüller space. On the other hand, fixing a symplectic basis of the homology, one can consider the normalized period matrix \( \Pi \) of holomorphic $1$-forms, where \( \det \Im \Pi \) is likewise a function on the Teichmüller space. These quantities are related by the following identity~\cite{Zograf1987,Kim2007}:
\[
\partial \bar{\partial} \log\left(\frac{\det'(\Delta)}{\det \Im \Pi}\right) = - \frac{\sqrt{-1}}{6 \pi} h,
\]
which identifies \( \log\left(\frac{\det'(\Delta)}{\det \Im \Pi}\right) \) as a potential function for the Weil--Petersson metric.

Theorem~\ref{thm:HQT_count} yields a combinatorial interpretation of this potential function as a weighted sum over \( g \)-quasi-trees (see Section~\ref{sec:normperiod}). Specifically, for a discrete Riemann surface, we have
\[
\log\left(\frac{\det'(\Delta)}{\det \Im \Pi}\right) = \log \sum_{\text{\( g \)-quasi-tree } T} \left( (\det \mathcal{T}_{I})^2 \prod_{e \in T} \mathsf{c}_{e} \right),
\]
where \( I = \{g+1, g+2, \dots, 2g\} \). This highlights a special role played by \( g \)-quasi-trees in the discrete setting.

Potential functions for the Weil--Petersson metric are not unique. For example, they also arise from the length of a uniformly distributed geodesic~\cite{Wolpert1986}, the renormalized volume of hyperbolic 3-manifolds~\cite{KS2008}, and the Loewner energy of planar Jordan curves in the context of Schramm--Loewner evolution~\cite{Wang2019}. The precise relationships among these various potential functions remain mysterious. Our results may provide new insights into a conjecture relating the number of spanning trees to the volume of hyperbolic 3-manifolds~\cite{CKL2019}.

For graphs embedded in a disk with vertices on the boundary, the discrete period matrix is analogous to the \emph{response matrix}, which maps the boundary values of a discrete harmonic function to the boundary values of its harmonic conjugate~\cite{Colin1994,Colin1996}. It is known that the minors of the response matrix can be expressed as weighted sums over certain subgraphs~\cite{Curtis1998,KW2009,Skopenkov2026}. In the context of closed surfaces, our discrete period matrix plays a role analogous to the response matrix, as also observed in~\cite{Lam2022}. Theorem~\ref{thm:HQT_count} provides an alternative answer to a question posed by Kenyon regarding the extension of combinatorial results from planar graphs to graphs embedded on closed surfaces~\cite[Section~5]{Kenyon2012}.

We note that the term ``quasi-trees'' has been adopted in various other contexts, of particular relevance to this article are those arising in matroid theory and topological graph theory~\cite{CMNR2019}, which we refer to as {\em ribbon-graph quasi-trees} to avoid confusion. A detailed comparison between our definition and these earlier notions is deferred to Section~\ref{sec:HQT_combin}, where we show that our definition properly includes them. A natural direction for future research is to establish parallel results for these alternative notions. For instance, it is a classical result that the collection of ribbon-graph quasi-trees forms a delta-matroid~\cite{Bouchet1989}. We have a corresponding result for homological quasi-trees:

\begin{theorem} \label{thm:delta_mat}
    The collection of homological quasi-trees forms a delta-matroid on the ground set $E$. Moreover, for each $0\leq k\leq 2g$, the collection of $k$-homological quasi-trees forms a matroid on the ground set $E$.
\end{theorem}

One may further ask about connections to invariants defined via ribbon-graph quasi-trees, such as the {Krushkal polynomial}~\cite{Krushkal2011, Butler2018} and the {critical group}~\cite{MMN2023}. However, these invariants appear to be quite distinct from the period matrix, and we therefore do not pursue them in this work.

We are grateful to Adrian Kassel for pointing out that Theorem 1.2 can also be deduced from a different perspective, as shown in the work by Kassel and Lévy~\cite{Kassel2022}.

\subsection*{Organization of the paper}

In Section~\ref{sec:background}, we recall some facts about the topology of surfaces and combinatorics of embedded graphs, as well as a few useful identities in linear algebra. In Section~\ref{sec:period}, we relate the period matrix of a discrete Riemann surface and the Laplacian of the corresponding graph. We prove Theorem~\ref{thm:HQT_count} and Theorem~\ref{thm:bijection} in Sections~\ref{sec:Omega_HQT} and~\ref{sec:duality}, respectively. Theorem~\ref{thm:Hess} is proven in Section~\ref{sec:crsf}, after an introduction to flat $\mathbb{C}$-bundles and their Laplacians. Section~\ref{sec:gHQT} is devoted to the special class of $g$-quasi-trees. The penultimate section, Section~\ref{sec:HQT_combin}, discusses the relation between homological quasi-trees and other well-studied combinatorial notions. Finally, Section~\ref{sec:example} presents a non-trivial example illustrating some of our results.

\section{Background} \label{sec:background}

This section collects the basic topological and combinatorial notions that are used throughout the paper. 
We recall facts about closed orientable surfaces, cellularly embedded graphs, and relevant linear algebra identities, which will provide the foundation for our later constructions and proofs.

\subsection{Notation and conventions}
Throughout this paper, we use the following notation.
For a matrix $M$ and index sets $I, J$, we denote by $M_{I,J}$ the submatrix of $M$ with rows indexed by $I$ and columns by $J$.
The transpose of a matrix $M$ is denoted by $M^t$.
For a set $X$, we denote its cardinality by $|X|$; in particular, $|V|$ and $|E|$ denote the number of vertices and edges of a graph $(V,E)$, respectively. For a function $f$ on a set of vertices (or edges), we use the notations $f(u)$ and $f_u$ interchangeably to denote the value of $f$ at $u$.
Finally, we identify linear operators between finite-dimensional vector spaces with their matrix representations with respect to the standard bases, freely switching between the operator and matrix perspectives.

\subsection{Topology of closed orientable surfaces}

We summarize some basic topological facts about surfaces that will be used throughout the paper. For further details, we refer the reader to~\cite{FarbMargalit}.

Let $S$ be a closed orientable surface of genus $g \geq 1$ with a chosen orientation. Throughout, we fix a smooth structure on $S$ and assume that the edges of the cell decomposition $(V,E,F)$, as well as all arcs considered, are smoothly embedded curves in $S$. For a pair of compact, oriented arcs $\gamma$ and $\tilde{\gamma}$ in $S$ that intersect transversely, we define the \emph{algebraic intersection} $\iota(\gamma,\tilde{\gamma}) \in \mathbb{Z}$ as the sum of the local intersection indices over all points of $\gamma \cap \tilde{\gamma}$. An intersection point $p \in \gamma \cap \tilde{\gamma}$ has index $+1$ if the ordered pair consisting of the tangent vector to $\gamma$ at $p$ followed by that of $\tilde{\gamma}$ is compatible with the chosen orientation of $S$, and index $-1$ otherwise. The algebraic intersection is skew-symmetric, that is,
\[
\iota(\gamma,\tilde{\gamma}) = - \iota(\tilde{\gamma},\gamma),
\]
and invariant under homological equivalence of loops.

It is well known that the first homology group of a closed oriented surface \( S \) of genus \( g \), \( H_1(S;\mathbb{Z}) \cong \mathbb{Z}^{2g} \), is a free abelian group of rank \( 2g \). There exist simple oriented loops \( \gamma_1, \gamma_2, \dots, \gamma_{2g} \) on \( S \) such that their homology classes \( [\gamma_1], [\gamma_2], \dots, [\gamma_{2g}] \) form a basis for \( H_1(S;\mathbb{Z}) \).
The algebraic intersection pairing \( \iota \) extends linearly to a bilinear form on \( H_1(S;\mathbb{Z}) \cong \mathbb{Z}^{2g} \), and defines a \emph{symplectic form}---a non-degenerate, skew-symmetric bilinear form. The oriented loops \( \gamma_i \) can be chosen such that their homology classes form a \emph{symplectic basis}, i.e., a basis for which the \emph{intersection matrix} \( \Omega := \big(\iota([\gamma_i], [\gamma_j])\big)_{i,j} \) is equal to
$\begin{pmatrix}
0 & \mathbb{I}_g \\
-\mathbb{I}_g & 0
\end{pmatrix}$.
A standard construction of a symplectic basis is to choose \( \gamma_i, \gamma_{g+i} \) as the ``latitudinal'' and ``longitudinal'' loops around the \( i \)-th ``hole'' of the surface \( S \), with compatible orientations.

Generalizing the convention in Section~\ref{sec:intro}, given a (not necessarily symplectic) basis of $H_1(S;\mathbb{Z})$ represented by oriented loops $\gamma_1,\ldots,\gamma_{2g}$, we may choose oriented loops $\gamma^1,\ldots,\gamma^{2g}$ such that $\iota(\gamma^i, \gamma_j)=\delta_{ij}$.

\subsection{Cellularly embedded graphs}\label{sec:cellemb}

Embedding a graph $G = (V, E)$ into a closed oriented surface $S$ means assigning each vertex to a point on $S$, and each edge to a simple arc connecting its endpoints, such that arcs only intersect at common endpoints. We assume the embedding is \emph{cellular}, meaning that the complement $S \setminus G$ decomposes into open discs called \emph{faces}. This structure defines a \emph{cellular decomposition} of the surface $S$.

Given such an embedding, the \emph{dual graph} $G^\star = (V^\star, E^\star)$ is defined as follows. Let $V^\star := \{ f^\star : f \in F \}$ and $E^\star := \{ e^\star : e \in E \}$. Place a vertex $f^\star$ in each face $f$ of $G$. For each edge $e \in E$ incident to faces $f_1$ and $f_2$, draw a dual edge $e^\star$ connecting $f_1^\star$ and $f_2^\star$, intersecting $e$ transversely at a single point in its interior. If $e$ is incident to the same face twice, that is, $f_1 = f_2$, then $e^\star$ is a loop edge. Throughout the article, the symbol $\star$ denotes dual objects corresponding to the dual graph.

The embedding $f: G \hookrightarrow S$ naturally induces a homomorphism on first homology groups, denoted by $f_*: H_1(G; \mathbb{Z}) \to H_1(S; \mathbb{Z})$. This map sends a $1$-cycle in the graph—formally, an integer linear combination of edges whose boundary is zero—to its image as a $1$-cycle in the surface. Since $G$ is embedded into $S$, these cycles can be interpreted as piecewise-smooth loops on $S$, and $f_*$ records their homology class in $S$. The induced homomorphism $f_*$ respects the additive structure of homology and preserves homological equivalence. In particular, cycles in $G$ that are null-homologous in $S$ lie in the kernel of $f_*$, and the image of $f_*$ reflects how the combinatorial structure of $G$ encodes information about the topology of $S$. 

Throughout the article, we assume that $\gamma_1,\gamma_2,\dots,\gamma_{2g}$ are oriented cycles on the embedded graph $(V,E)$ (all cycles, without the prefix ``$1$-'' as above, in this article are assumed to be simple) whose homology classes form a symplectic basis. Moreover, the respective $\gamma^1,\gamma^2,\dots,\gamma^{2g}$ are represented by cycles of $G^\star$.  We only consider intersection between $G$ and $G^\star$; since an edge of $G$ and a dual edge of $G^\star$ intersect transversely, the convention in the last section applies. These choices streamline both the exposition and the arguments, even though the results remain valid for general cases\footnote{In general, one can represent any homology class by a closed walk on the graph or its dual, in which case a set-theoretic intersection point may be crossed multiple times on each walk. So in the definition of $\iota$, the local intersection index should take account of such multiplicity in the obvious way.}.

We often describe an embedded graph using purely combinatorial data.

\begin{definition}[\cite{GrossTucker, Mondello2009}]
    A {\em rotation system} or {\em (orientable) ribbon graph} is a (finite, connected) graph together with a cyclic permutation of the set of half-edges
at each vertex, where each edge is regarded as a pair of half-edges corresponding to its two endpoints.
\end{definition}

Every graph embedded on an orientable surface inherits a ribbon structure. Conversely, every ribbon graph determines a \emph{cellular} embedding of the graph into an orientable surface (up to isotopy) by thickening the graph into a surface with boundary according to the cyclic ordering at each vertex, and then capping off each boundary component with a disc. 

Moreover, deletion and contraction can be defined naturally for ribbon graphs in a way that is compatible with standard graph operations. In particular, we can speak of the intrinsic embedding of a (spanning) subgraph of an embedded graph, and interpret combinatorially the fact that contracting a spanning tree does not change the topological type of the ambient surface.

\begin{definition}
    A {\em tree-cotree decomposition} of an embedded graph $G$ is a partition $T_0\sqcup R\sqcup C_0$ of the edge set of $G$ such that $T_0$ is a spanning tree of $G$ and ${C_0^\star}$ is a spanning tree of the dual graph $G^\star$.
\end{definition}

Given a tree-cotree decomposition $T_0 \sqcup R \sqcup C_0$ of an embedded graph $G$, we may contract the edges in $T_0$ and delete the edges in $C_0$ (which corresponds to contracting their dual edges in $G^\star$). The resulting graph, denoted $G/T_0 \!\setminus\!C_0$, has a single vertex and a single face, with edge set $R$ consisting of $2g$ edges as loops. It remains embedded on a surface of the same genus~\cite{Eppstein_TRC}.

\begin{lemma} \label{lem:TLC_homology}
    With the above setup, the loops of $G/T_0\!\setminus\!C_0$, viewed as $1$-homology classes by choosing an arbitrary orientation, form a basis of $H_1(S;\mathbb{Z})$. Dually, the loops of $G^\star/{C_0^\star}\!\setminus\!{T_0^\star}$ also form a basis of $H_1(S;\mathbb{Z})$. In fact, each loop in $G/T_0\!\setminus\!C_0$ intersects its corresponding dual loop in $G^\star/{C_0^\star}\!\setminus\!{T_0^\star}$ transversely at a single point and is disjoint from all other loops.
\end{lemma}

\begin{proof}
    The surface can be built by attaching the unique face as a $2$-cell to the loops in $R$, which form the $1$-skeleton. Since each edge in $R$ appears twice in the boundary of the face with opposite orientations, the boundary map $\partial_2$ is the zero map with trivial image. Hence, the first homology group $H_1=\ker\partial_1/{\rm im\!\ } \partial_2$ is freely generated by $R$; so $H_1(S;\mathbb{Z}) \cong \mathbb{Z}^{2g}$ with basis given by the loops in $R$.

    The dual statement follows from the fact that, by construction, each dual edge $e^\star \in {R^\star}$ intersects exactly one primal edge $e \in R$ transversely.
\end{proof}

In case one wants to work with the original embedded graph $G$ instead of $G/T_0\!\setminus\!C_0$, for every edge $e\in R$, there is a unique cycle, known as the {\em fundamental cycle}, contained in $T_0\cup\{e\}$. Upon fixing an arbitrary orientation of the cycle (which we assume to be the case in all our discussions of fundamental cycles in this article), it represents a homology class in $H_1(G;\mathbb{Z})$ and thus in $H_1(S;\mathbb{Z})$ via pushforward. There is an obvious bijection between these fundamental cycles (and the homology classes thereof) and the oriented loops of $G/T_0\!\setminus\!C_0$ in Lemma~\ref{lem:TLC_homology}.

Analogously to the intersection matrix $\mathcal{T}_{T,I}$, for the subgraph 
$\overline{T}^{\star}:=(E\setminus T)^\star$ in the dual and $I' \subset \{1,2,\dots,2g\}$ with $|I'|=2g-k$, we define the matrix $\mathcal{T}^\star_{T,I'}$ as follows. 
Write $I' = \{i'_1,i'_2,\dots,i'_{2g-k}\}$ with $i'_1 < i'_2 < \dots < i'_{2g-k}$. 
Let $\mu^1, \dots, \mu^{2g-k}$ form a basis of $H_1(\overline{T}^{\star};\mathbb{Z})$. 
Then the $(2g-k) \times (2g-k)$ matrix $\mathcal{T}^\star_{T,I'}$ is defined by
\[
(\mathcal{T}^\star_{T,I'})_{r,s} := \iota(f_*\mu^s, \gamma_{i'_r}).
\]
We omit the reference to $T$ when it is clear from context.

\subsection{Some identities in linear algebra}

We recall several classical identities in linear algebra that will be used in our later arguments. 

\begin{proposition}[Schur complement {\cite[Section~3.1]{Prasolov}}]\label{prop:Schur}
    Given a square block matrix $\begin{pmatrix} P & Q \\ R & S \end{pmatrix}$, where $P$ and $S$ are square submatrices and $P$ is invertible, we have
\[
\det \begin{pmatrix}
	P & Q \\
	R & S 
\end{pmatrix} = \det P \cdot \det(S - R P^{-1} Q).
\]
\end{proposition}

\begin{proposition} [Cauchy--Binet formula {\cite[Theorem~2.3]{Prasolov}}] \label{prop:CauchyBinet}
    Let $A, B$ be $m \times n$ matrices with $m \leq n$. Index the columns of $A$ and $B$ by $\{1, \ldots, n\}$. Then
\[
\det(AB^t) = \sum_{\substack{I \subset \{1, \ldots, n\} \\ |I| = m}} \det A_I \cdot \det B_I^t,
\]
where $A_I$ and $B_I$ are the $m \times m$ submatrices of $A$ and $B$ whose columns are indexed by $I$.
\end{proposition}

\begin{proposition} [Jacobi equality {\cite[Theorem~2.5.2]{Prasolov}}] \label{prop:Jacobi_adj}
    Let $A = (a_{i,j})$ be an $n \times n$ nonsingular matrix whose rows and columns are indexed by $\{1,\ldots,n\}$. Let $(\widetilde{a_{i,j}})$ denote the transpose of the adjugate matrix of $A$, i.e.\ $(\widetilde{a_{i,j}}) = (\adj A)^t$. Let $\sigma$ and $\pi$ be permutations of $\{1,\ldots,n\}$, and define $i_k = \sigma(k)$ and $j_k = \pi(k)$ for $1 \leq k \leq n$. Then for any $1 \leq p < n$,
    $$
        \det \begin{pmatrix}
            a_{i_1,j_1} & \ldots & a_{i_1, j_p} \\
            \vdots & \ddots & \vdots \\
            a_{i_p,j_1} & \ldots & a_{i_p,j_p}
        \end{pmatrix}
        = \frac{(-1)^\epsilon}{(\det A)^{p-1}} \cdot \det \begin{pmatrix}
            \widetilde{a_{i_{p+1},j_{p+1}}} & \ldots & \widetilde{a_{i_{p+1},j_n}} \\
            \vdots & \ddots & \vdots \\
            \widetilde{a_{i_n,j_{p+1}}} & \ldots & \widetilde{a_{i_n,j_n}}
        \end{pmatrix},
    $$
    where $(-1)^\epsilon$ is the sign of the permutation $\sigma \circ \pi$.
\end{proposition}

\section{Period matrices} \label{sec:period}

The period matrix, previously introduced for discrete Riemann surfaces, will here be expressed in terms of the discrete Laplacian. 
In this section, we prove Proposition~\ref{prop:Lop} by extending the approach of~\cite{Lam2024torus}, which treats the torus case ($g=1$), to surfaces of arbitrary genus $g$.

We recall that a \emph{discrete $1$-form} is a function on the oriented edges $\omega : \vec{E} \to \mathbb{R}$ such that $\omega_{uv} = -\omega_{vu}$. It is said to be \emph{closed} if, for every face $\phi \in F$ whose boundary $\partial \phi$ is composed of oriented edges, we have $\sum_{uv \in \partial \phi} \omega_{uv} = 0$.
It is \emph{co-closed} if it is a closed $1$-form with respect to the dual decomposition; that is, for every vertex $v \in V$, $\sum\omega_{uv} = 0$, where the sum ranges over all edges $uv$ incident to $v$.
Recall that $\mathsf{c}$ denotes the positive edge weights. For every $1$-form $\omega$, there is an associated $1$-form $\star\omega$ defined by
\[
(\star\omega)_{uv} := \mathsf{c}_{uv} \, \omega_{uv}.
\]
This map $\omega \mapsto \star\omega$ serves as a discrete analogue of the Hodge star operator.  
A $1$-form $\omega$ on the primal decomposition $(V, E, F)$ is called \emph{harmonic} if it is closed and its Hodge dual $\star\omega$ is co-closed.

We first start with the Riemann bilinear relation for the period matrix. For any two discrete harmonic $1$-forms $\omega$ and $\eta$, we can consider a symmetric bilinear form 
\[
\sum \mathsf{c}_{uv} \omega_{uv} \eta_{uv} = 	\sum  \omega({\star}\eta) = \sum (\star\omega ) \eta.
\]

\begin{proposition}[\cite{Bobenko2016}] \label{prop:stokes}
	Suppose $\omega, \eta:\vec{E} \to \mathbb{R}$ are $1$-forms such that $\omega$ is closed with periods $A \in \mathbb{R}^{2g}$ while $\eta$ is co-closed with periods $B \in \mathbb{R}^{2g}$ with respect to the same homology basis. Then
	\begin{align} \label{eq:bilinearform}
		\sum \omega_{uv}  \eta_{uv} =  A^t \Omega B =  -B^t \Omega A. 
	\end{align}
\end{proposition}
\begin{corollary}
		Suppose $\omega,\eta:\vec{E} \to \mathbb{R}$ are harmonic $1$-forms with periods $A$ and $B$, respectively. Then \begin{align}\label{eq:bilinearform2}
		\sum \mathsf{c}_{uv} \omega_{uv} \eta_{uv} = A^t \Omega L B = B^t \Omega L A. 
	\end{align}
\end{corollary}This implies that $\Omega L$ is symmetric. Furthermore, whenever $\omega=\eta$ is not vanishing everywhere, the quantity in Equation \eqref{eq:bilinearform2} is positive and hence $\Omega L$ is positive definite.

We now represent the Laplacian using the discrete differential operators. To this end, we define the incidence matrix $d \in \mathbb{R}^{|E| \times |V|}$. Fix an arbitrary orientation for each edge $e$, denoting its head and tail by $e_+$ and $e_-$, respectively. Then the operator $d: \mathbb{R}^{V} \to \mathbb{R}^E$ is defined by
\[
dh(e) = h(e_+) - h(e_-).
\]
Observe that for any function $h: V \to \mathbb{R}$, the differential $dh$ defines a closed discrete $1$-form. A discrete $1$-form $\omega$ is said to be \emph{exact} if there exists a function $h: V \to \mathbb{R}$ such that $\omega = dh$.

We further define
\[
\Delta := d^{t} \mathsf{C} d,
\]
where $\mathsf{C}$ is the $E \times E$ diagonal matrix with entries $\mathsf{C}_{e,e} = \mathsf{c}_e$, the corresponding edge weight, and $d^t$ denotes the transpose of $d$. One can verify that the operator $\Delta : \mathbb{R}^V \to \mathbb{R}^V$ satisfies, for every vertex $v \in V$,
\[
(\Delta h)_v = \sum \mathsf{c}_{uv} (h_v - h_u),
\]
where the sum ranges over all edges $uv$ incident to $v$. The matrix $\Delta$ is called the \emph{discrete Laplacian}. Notably, the definition of $\Delta$ is independent of the choice of edge orientations, unlike that of the incidence matrix $d$. Furthermore, $\Delta$ is positive semi-definite, and its kernel consists of constant functions.

By interpreting the linear map $d: \mathbb{R}^{V} \to \mathbb{R}^E$ as a matrix, we observe that each column vector of $d$ defines a function on oriented edges, i.e., an exact discrete $1$-form. Any collection of $|V|-1$ column vectors forms a basis for the space of exact $1$-forms on the cell decomposition $G = (V, E, F)$. 

To extend this to a basis for the space of closed $1$-forms, we introduce $2g$ closed $1$-forms $m_i : \vec{E} \to \mathbb{R}$ with prescribed periods. Specifically, for $i, j \in \{1, \dots, 2g\}$, we require that the integral of $m_i$ along each cycle $\gamma_j$ satisfies
\[
\sum_{e \in \gamma_j} m_i(e) = \delta_{ij}.
\]

Such closed $1$-forms can be constructed as follows. Recall the collection $\gamma^1, \gamma^2, \dots, \gamma^{2g}$ as defined in Section~\ref{sec:background}, which satisfy $\iota(\gamma^i,\gamma_j)=\delta_{ij}$.

We then define $m_i : \vec{E} \to \mathbb{R}$ by
\[
m_i(e) := \iota(\gamma^i, e)
\]
for each oriented edge $e$ of the primal graph. Each $m_i$ is a closed $1$-form satisfying the desired condition:
\[
\sum_{e \in \gamma_j} m_i(e) = \iota(\gamma^i, \gamma_j) = \delta_{ij}.
\]

In this way, every closed $1$-form $\omega$ with periods  $A=(a_1,\dots,a_{2g})^t \in \mathbb{R}^{2g}$ can be expressed as
\begin{align}\label{for:M}
    \omega = dh+  \begin{pmatrix}
	\vert & &\vert  \\
	m_1 & \dots &m_{2g} \\
	\vert & &\vert
\end{pmatrix} \begin{pmatrix}
	a_1 \\ \vdots \\ a_{2g}
\end{pmatrix} =:  d h + M \begin{pmatrix}
	a_1 \\ \vdots \\ a_{2g}
\end{pmatrix} 
\end{align}
for some $h \in \mathbb{R}^{V}$ unique up to an additive constant. The closed $1$-form $\omega$ is harmonic if
\[
0 = d^t \mathsf{C} \omega =  d^t \mathsf{C} dh +  d^t \mathsf{C} M A = \Delta h +  d^t \mathsf{C} M A.
\]
To obtain a unique solution for $h$, we can pick a vertex $o$ and demand $h_o=0$. We write $d_{\bar{o}}$ as the submatrix of $d$ with the column corresponding to vertex $o$ removed. We also write $\Delta_{\bar{o} \bar{o}}$ as the submatrix of $\Delta$ with the column and the row corresponding to vertex $o$ removed; it is straightforward to see that $\Delta_{\bar{o} \bar{o}}=d_{\bar{o}}^t\mathsf{C}d_{\bar{o}}$. One deduces that $\Delta_{\bar{o} \bar{o}}$ is invertible, since its determinant counts spanning trees by Kirchhoff's Matrix-Tree Theorem and hence positive. The values of $h$ at vertices other than $o$ can be obtained via
\begin{align}\label{eq:inversef}
	h_{\bar{o}} = - \Delta_{\bar{o} \bar{o}}^{-1} d_{\bar{o}}^t \mathsf{C} M A.
\end{align}
Now we can express the linear map $L: \mathbb{R}^{2g} \to \mathbb{R}^{2g}$ in terms of the edge weights. 
\begin{proposition}\label{prop:Lop}
	\[
		\Omega L =   M^t \mathsf{C}M - M^t \mathsf{C} d_{\bar{o}} \Delta_{\bar{o} \bar{o}}^{-1} d_{\bar{o}}^t \mathsf{C} M.
	\]
\end{proposition}
\begin{proof}
	To compute the formula, we need further notations on the dual graph. Since the orientation of the primal edges is chosen, it naturally induces an orientation of the dual edges. Namely, a dual edge $e^\star$ is oriented from the right face of $e$ to the left. We then define the incidence matrix $d^{\star}: \mathbb{R}^F \to \mathbb{R}^E$ similarly. The columns of $d^{\star}$ span the space of exact $1$-forms on the dual graph. We further define a $|E|{\times}2g$ matrix $M_{\star}$ such that its columns $m_{\star i}$ represent closed $1$-forms on the dual graph having non-trivial periods 
\begin{align*}
    \sum_{e\in\gamma_j} m_{\star i}(e)= \delta_{ij}.
\end{align*}
    
    Let $\omega$ be a harmonic $1$-form with period vector $A \in \mathbb{R}^{2g}$, and let $A^\star \in \mathbb{R}^{2g}$ be the period vector of the dual harmonic form $\star\omega$ on the dual graph. Note that $\omega$ is closed and $\star\omega$ is co-closed (on the primal graph). So we have
	\[
	\mathsf{C} d h + \mathsf{C} M A  = \mathsf{C} \omega = \star\omega = d^{\star} \tilde{h} + M_{\star} A^{\star}
	\]
	for some $h:V \to \mathbb{R}$ with $h_o = 0$ and $\tilde{h}:F \to \mathbb{R}$. Multiplying $M^t$ to both sides from the left yields
	\[
	M^t \mathsf{C} d h + M^t \mathsf{C} M A  = (M^t d^{\star}) \tilde{h} + M^t M_{\star} A^{\star} = 0 + \Omega A^{\star},
	\]
	where Proposition~\ref{prop:stokes} is applied to  $M^t d^{\star}$ and $M^t M_{\star}$ as the rows of $M^t$ represent closed $1$-forms on the primal  graph while the columns of $d^{\star}$ and $M_{\star}$ represent co-closed $1$-forms. Because $h_{o}=0$, we have from Equation \eqref{eq:inversef}
	\begin{align*}
		A^{\star} &=  \Omega^{-1}  (M^t \mathsf{C} d_{\bar{o}} h_{\bar{o}} + M^t \mathsf{C}M A)  \\  &= \Omega^{-1}  (- M^t \mathsf{C} d_{\bar{o}} \Delta_{\bar{o} \bar{o}}^{-1} d_{\bar{o}}^t \mathsf{C} M + M^t \mathsf{C}M ) A
	\end{align*}
	and thus 
    \begin{align*}
        \Omega L &=   M^t \mathsf{C}M - M^t \mathsf{C} d_{\bar{o}} \Delta_{\bar{o} \bar{o}}^{-1} d_{\bar{o}}^t \mathsf{C} M. \qedhere
    \end{align*}
\end{proof}

\section{Period matrices as weighted sum over quasi-trees} \label{sec:Omega_HQT}

This section is devoted to proving Theorem~\ref{thm:HQT_count}, using the interpretation of the period matrix developed in the previous section.

\begin{proof}[Proof of Theorem~\ref{thm:HQT_count}]

Fix an integer $1 \leq k \leq 2g$ and index sets $I,J \subset \{1,2,\dots,2g\}$ such that $|I|=|J|=k$. Denote by $M_I, M_J$ the submatrices of $M$ consisting of the columns indexed by $I$ and $J$, respectively. By Proposition~\ref{prop:Lop},
\[
	(\Omega L)_{I,J} =   M_I^t \mathsf{C}M_J - M_I^t \mathsf{C} d_{\bar{o}} \Delta_{\bar{o} \bar{o}}^{-1} d_{\bar{o}}^t \mathsf{C} M_J.
\]
Applying Proposition~\ref{prop:Schur} yields
\begin{align*}
\det  	\Delta_{\bar{o} \bar{o}} \det (\Omega L)_{I,J} &= \det \begin{pmatrix}
 	\Delta_{\bar{o} \bar{o}} &  d_{\bar{o}}^t \mathsf{C} M_J \\ M_I^t \mathsf{C} d_{\bar{o}} & M_I^t \mathsf{C}M_J 
 \end{pmatrix} \\ &= 	\det  \left( \begin{pmatrix}
 d_{\bar{o}} & M_I
 \end{pmatrix}^t \mathsf{C}   \begin{pmatrix}
 d_{\bar{o}} & M_J
 \end{pmatrix} \right).
\end{align*}

By Proposition~\ref{prop:CauchyBinet},
\[
\det  \left( \begin{pmatrix}
	d_{\bar{o}} & M_I
\end{pmatrix}^t \mathsf{C}   \begin{pmatrix}
	d_{\bar{o}} & M_J
\end{pmatrix} \right) = \sum_{T} \det  \begin{pmatrix}
d_{T,{\bar{o}}} & M_{T,I}
\end{pmatrix}^t \det \mathsf{C}_{T,T} \det  \begin{pmatrix}
d_{T,{\bar{o}}} & M_{T,J}
\end{pmatrix},
\]
where the sum is over all subsets $T$ consisting of $(|V|-1+k)$ edges, and $\begin{pmatrix}	d_{T,{\bar{o}}} & M_{T,I} \end{pmatrix}$ and $\begin{pmatrix}	d_{T,{\bar{o}}} & M_{T,J} \end{pmatrix}$ are the submatrices of $\begin{pmatrix} d_{\bar{o}} & M_I \end{pmatrix}$ and $\begin{pmatrix} d_{\bar{o}} & M_J \end{pmatrix}$ consisting of the rows corresponding to $T$, respectively.

Since $\det \mathsf{C}_{T,T} = \prod_{uv \in T} \mathsf{c}_{uv}$, it remains to show that (1) a term of the form
\[
\det \begin{pmatrix} d_{T,\bar{o}} & M_{T,I} \end{pmatrix} \det \begin{pmatrix} d_{T,\bar{o}} & M_{T,J} \end{pmatrix}
\]
is nonzero only if $T$ is a $k$-quasi-tree, and (2) in that case, it equals $\det \mathcal{T}_I \det \mathcal{T}_J$.

\smallskip

We first prove (2). Suppose $T$ is a $k$-quasi-tree. By definition, it is connected, so it contains a spanning tree $T_0$. Let $R := T \setminus T_0$. Recall that the image of the (oriented) fundamental cycles $\mu_1, \dots, \mu_{|R|}$ associated to $T_0$, under the pushforward map $f_*$, form a basis of $H_1(T;\mathbb{Z})$. We shall use it as the basis in the definition of $\mathcal{T}_I$.

Additionally, we can assume that each $\gamma^i$ ($i \in I \cup J$) has no intersection with $T_0$, by replacing it with a homologous representative. This is possible since $T_0$ is contractible, and such a modification changes $m_i$ only by a linear combination of the columns of $d_{\bar{o}}$, and hence does not affect the determinant.

Under these assumptions, the matrix $\begin{pmatrix} d_{T,\bar{o}} & M_{T,I} \end{pmatrix}$ takes the block form
\[
\begin{pmatrix}
d_{T_0,\bar{o}} & 0 \\
* & M_{R,I}
\end{pmatrix},
\]
where $d_{T_0,\bar{o}}$ is the incidence matrix associated with the spanning tree $T_0$ and with the column corresponding to the root vertex $o$ removed; $d_{T_0,\bar{o}}$ is a unimodular matrix. Furthermore, we have
\[
(M_{R,I})_{r,i} = \iota(\gamma^{i}, f_*\mu_r) = (\mathcal{T}_I)_{i,r} \quad \text{for } i \in I \text{ and } 1 \le r \le |R|.
\]
It follows that
\[
\det \begin{pmatrix} d_{T,\bar{o}} & M_{T,I} \end{pmatrix} = \det d_{T_0,\bar{o}} \cdot \det M_{R,I} = \pm \det \mathcal{T}_I,
\]
where the sign depends only on the choice of $T_0$ and $\bar{o}$, and cancels out when multiplied with the corresponding determinant for $J$.

\smallskip

Now we prove (1). If $T$ is not a $k$-quasi-tree, then one of the following holds: 
\begin{enumerate}
    \item[(i)] $T$ is disconnected. In this case, the sum of the columns of $d_{T,\bar{o}}$ corresponding to the vertices of any connected component not containing the root $o$ is zero, implying that $d_{T,\bar{o}}$ has linearly dependent columns.
    
    \item[(ii)] $T$ is connected. As before, we consider the block form $\begin{pmatrix}
    d_{T_0,\bar{o}} & 0 \\
    * & M_{R,I}
    \end{pmatrix}$, where $T_0$ is a spanning tree of $T$ and $R = T \setminus T_0$. In this case, the pushforward map $f_*: H_1(T; \mathbb{Z}) \to H_1(S; \mathbb{Z})$ has rank less than $k$ by definition, so there exists a linear dependence among $f_*\mu_r$'s. In particular, $M_{R,I}$ has linearly dependent rows thus linearly dependent columns, the latter can be said for $\begin{pmatrix} d_{T,\bar{o}} & M_{T,I} \end{pmatrix}$ as well.
\end{enumerate}
In either case, the determinant of $\begin{pmatrix} d_{T,\bar{o}} & M_{T,I} \end{pmatrix}$ is zero.
\end{proof}

In general, the determinants $\det \mathcal{T}_I$ can serve as interesting invariants in their own right; see Sections~\ref{sec:duality} and~\ref{sec:interpret} for further discussion. However, in the special case of maximal homological quasi-trees, i.e., when $k = 2g$, these determinants are always $\pm 1$, and Theorem~\ref{thm:HQT_count} specializes to Corollary~\ref{cor:det}. The following statement corresponds to the $k = 0$ case of Theorem~\ref{thm:bijection}. Since the linear-algebraic argument in Section~\ref{sec:duality} does not apply in this case, we include it here.

\begin{proposition} \label{prop:MHQT_TU}
    Let $T$ be a subgraph consisting of $|V|-1+2g$ edges and set $I = \{1, 2, \ldots, 2g\}$.
    The determinant $\det \mathcal{T}_I$ is either $0$—in which case $T$ is not a homological quasi-tree—or $\pm 1$.
\end{proposition}

\begin{proof}
    Following the proof of Theorem~\ref{thm:HQT_count}, if $T$ is not a $2g$-quasi-tree, then $\det \mathcal{T}_I = 0$. If $T$ is indeed a $2g$-quasi-tree, then by the tree-cotree decomposition, we have $T = T_0 \sqcup R$. Moreover, after contracting $T_0$, the loops in $R$ form a $\mathbb{Z}$-basis of $H_1(S;\mathbb{Z})$ by Lemma~\ref{lem:TLC_homology}. Therefore, there exists a unimodular change-of-basis matrix $N$ between the loops in $R$ and the canonical basis $\gamma_1, \ldots, \gamma_{2g}$. By the linearity of the intersection pairing, we have $\mathcal{T}_I=N$, and hence its determinant is $\pm 1$.
\end{proof}

\section{Duality of homological quasi-trees} \label{sec:duality}

This section develops a duality between $k$-quasi-trees of a graph $G$ and $(2g-k)$-quasi-trees of its dual $G^\star$. Using this correspondence, we deduce Theorem~\ref{thm:bijection} and derive Corollary~\ref{cor:HQT_duality_count} relating sums over quasi-trees in the primal and dual graphs.

\begin{proposition} \label{prop:HQT_duality}
    Let $T$ be a $k$-quasi-tree of $G$. Then the complement edges ${\overline{T}^{\star}}:=\{e^\star: e\in E\setminus T\}$ form a $(2g-k)$-quasi-tree of $G^\star$. In fact, there exists a tree-cotree decomposition $E=T_0\sqcup R\sqcup C_0$ such that $T_0\subset T\subset E\setminus C_0$. 
    Conversely, given a tree-cotree decomposition $E=T_0\sqcup R\sqcup C_0$, any subgraph of the form $T_0\sqcup R'$ with $R'\subset R$ is a homological quasi-tree.
\end{proposition}

\begin{proof}    
    We first claim that $(V^\star,{\overline{T}^{\star}})$ is a connected subgraph of $G^\star$. Suppose not. Then $G^\star\setminus T^{\star}$ has a proper connected component whose vertex set is $V_0^\star\subset V^\star$, in particular the dual edges between $V_0^\star$ and $V^\star\setminus V_0^\star$ form a cut that is contained in $T^{\star}$. The primal picture is that $T$ contains the boundary of a proper subset of faces, such boundary as a $1$-cycle of $T$ represents a non-trivial class of $H_1(T;\mathbb{Z})$ but its image in $H_1(S;\mathbb{Z})$ is trivial. This is a contradiction to the assumption that $H_1(T;\mathbb{Z})\rightarrow H_1(S;\mathbb{Z})$ is an injective map (of rank $k$).

    Since both $T$ and ${\overline{T}^{\star}}$ are connected, we can respectively pick a spanning tree $T_0\subset T, {C_0^\star}\subset {\overline{T}^{\star}}$ from each of them. As $T_0\sqcup R\sqcup C_0$ is a tree-cotree decomposition, we may consider $G':=G/T_0\!\setminus\!C_0$. $H_1(T;\mathbb{Z})$ is generated by $T\cap R$ (viewed as loops of $G'$) and $H_1({\overline{T}^{\star}};\mathbb{Z})$ is generated by ${R^\star}\cap {\overline{T}^{\star}}$, which by Lemma~\ref{lem:TLC_homology} is of rank $2g-k$.
\end{proof}

\begin{proof}[Proof of Theorem~\ref{thm:bijection}]
    If $T$ is a $k$-quasi-tree, then ${\overline{T}^{\star}}$ is a $(2g-k)$-quasi-tree by the above proposition. Hence its homology $H_1(\overline{T}^{\star};\mathbb{Z})$ is of rank $2g-k$ and we can pick a basis $\mu^1,\ldots,\mu^{2g-k}$ for it, whose image under $f_*$ remains linearly independent. List the elements of $\overline{I}=\{1,\ldots,2g\}\setminus I$ as $\{i'_1,\ldots,i'_{2g-k}\}$. We can then form the matrix $\mathcal{T}^\star_{\overline{I}}$ in Theorem~\ref{thm:bijection} whose $(r,s)$-entry is $\iota(f_*\mu^s,\gamma_{i'_r})$.

    The $k=0$ (respectively, $k=2g$) case follows from Proposition~\ref{prop:MHQT_TU}, so we consider $1\leq k\leq 2g-1$. Given a $k$-quasi-tree $T$, we choose a tree-cotree decomposition $E=T_0\sqcup R\sqcup C_0$ as in Proposition~\ref{prop:HQT_duality}. By contracting $T_0$ and deleting $C_0$, we may assume $G$ and $G^\star$ each have one vertex, one face, and $2g$ edges $R\leftrightarrow {R^\star}$. List the edges of $R$ as $e_1,\ldots,e_{2g}$ such that $e_1,\ldots,e_k\in T, e_{k+1},\ldots, e_{2g}\in R\setminus T$, and denote their dual edges in $R^\star$ as $e^\star_1,\ldots,e^\star_{2g}$. Consider the $2g\times 2g$ matrix $\mathcal{R}$ whose $(i,j)$-entry is $\iota(\gamma^i, e_j)$. By Proposition~\ref{prop:MHQT_TU}, $\det \mathcal{R}=\pm 1$, where $\det\mathcal{T}_I$ is $\det\mathcal{R}_{I,T}$ by definition. Similarly, consider the $2g\times 2g$ matrix $\mathcal{R}^\star$ whose $(i,j)$-entry is $\iota(e^\star_j,\gamma_i)$. Then $\det\mathcal{T}^\star_{\overline{I}}=\det\mathcal{R}^\star_{\overline{I},R\setminus T}$.

    Note that reordering the edges corresponds to column operations on $\mathcal{R},\mathcal{R}^\star$, which only introduces an extra sign factor. Moreover, the sign only depends on $T$ (and hence canceled out upon multiplying to the other term corresponding to $J$ in the theorem statement).

    Since $\delta_{j,j'}=\iota(e^\star_{j'},e_j)=\sum_{i}\iota(e^\star_{j'},\gamma_i)\cdot\iota(\gamma^i,e_j)=\sum_i \mathcal{R}^\star_{i,j'}\mathcal{R}_{i,j}$, we have $\mathcal{R}^\star=(\mathcal{R}^{-1})^t=\det\mathcal{R}\cdot (\adj \mathcal{R})^t$.
    By Proposition~\ref{prop:Jacobi_adj}, $\det\mathcal{R}_{I,T}=(-1)^\epsilon(\det\mathcal{R})^{2g-k-1}\det(\adj \mathcal{R})^t_{\overline{I},R\setminus T}=(-1)^\epsilon\det\mathcal{R}\det\mathcal{R}^\star_{\overline{I},R\setminus T}$, here $\epsilon$ is related to the ordering of rows, which only depends on the fixed ordering of $I$.
\end{proof}

The following statement is a considerably weaker corollary of the above individual equality of terms: it can be deduced directly from the observation that the left-hand side of the dual graph counterpart of Theorem~\ref{thm:HQT_count} (or Proposition~\ref{prop:Lop}) uses $\Omega^t L^{-1}$. Here $L^{-1}$ is the discrete period matrix mapping the periods of harmonic $1$-forms on the dual graphs to those on the primal graph.

\begin{corollary}\label{cor:HQT_duality_count}
    Given $0 \leq k \leq 2g$ and $I, J \subset \{1, 2, \dots, 2g\}$ with $|I| = |J| = k$, we have
    \[
    \frac{
    \displaystyle\sum_{T}  
    \left( \det \mathcal{T}_{I} \det \mathcal{T}_{J} \prod_{e \in T} \mathsf{c}_{e} \right)
    }{
    \displaystyle\sum_{S}  
    \left( \prod_{e \in S} \mathsf{c}_{e} \right)
    }
    =
    \pm
    \frac{
    \displaystyle\sum_{{\overline{T}^{\star}}}  
    \left( \det \mathcal{T}^\star_{\overline{I}} \det \mathcal{T}^\star_{\overline{J}} \prod_{e \in T^{\star}} \mathsf{c}^{-1}_{e} \right)
    }{
    \displaystyle\sum_{{\overline{S}^{\star}}}  
    \left( \prod_{e \in S^{\star}} \mathsf{c}^{-1}_{e} \right)
    },
    \]
    where the sums in the numerators and denominators are taken over $k$-quasi-trees $T$ and $0$-quasi-trees $S$ of the primal graph $G$, and over $(2g - k)$-quasi-trees ${\overline{T}^{\star}}$ and $2g$-quasi-trees ${\overline{S}^{\star}}$ of the dual graph $G^\star$, respectively. Moreover, the sign factor depends only on $I$ and $J$.
\end{corollary}

\begin{proof}
    Each term $\det \mathcal{T}_{I} \det \mathcal{T}_{J}  \prod_{e \in T} \mathsf{c}_{e}$ corresponding to a $k$-quasi-tree $T$ of $G$ is equal to $\pm\prod_{e} \mathsf{c}_{e}\cdot\det \mathcal{T}^\star_{\overline{I}} \det \mathcal{T}^\star_{\overline{J}}  \prod_{e \in {\overline{T}^{\star}}} \mathsf{c}^{-1}_{e}$, where ${\overline{T}^{\star}}$ is a $(2g-k)$-quasi-tree of $G^\star$, by Theorem~\ref{thm:bijection} and vice versa. Hence there is a bijection between the terms in the numerator of the left-hand side and those of the right-hand side, up to a factor $\pm\prod_{e} \mathsf{c}_{e}$ whose sign is independent of $T$. The same can be said for the denominators, e.g., by setting $k=0$.
\end{proof}

\section{Laplacian on flat $\mathbb{C}$-bundle}\label{sec:crsf}

This section is devoted to proving Theorem~\ref{thm:Hess}, which expresses the Hessian of the bundle Laplacian polynomial at the trivial connection in terms of the period matrix.

Given a graph $G = (V, E)$ embedded on a closed orientable surface $S$ of genus $g$, the enumeration of certain subgraphs has appeared in a variety of contexts. A setting closely related to ours arises from the study of the Laplacian associated to a flat $\mathbb{C}$-bundle on the embedded graph, which we discuss in this section. We follow the exposition in ~\cite[Section 5.1]{Kenyon2012}. 

A {\em discrete $\mathbb{C}$-bundle} with connection $(W,\phi)$ on a graph consists of the following data:
\begin{enumerate}
    \item a $1$-dimensional complex vector space $W_u \cong \mathbb{C}$ for every vertex $u$ and
    \item a linear isomorphism $\phi_{uv}:W_u \to W_v$ for every edge $uv$ oriented from vertex $u$ to $v$ satisfying $\phi_{uv}=\phi_{vu}^{-1}$.
\end{enumerate}
Since the vector spaces associated with the vertices are $1$-dimensional, we identify each $\phi_{uv}$ with a nonzero complex number.
The connection is {\em flat} on the embedded graph if, for every contractible loop $\gamma$ on the surface consisting of oriented edges, we have
\[
\prod_{uv \in \gamma} \phi_{uv} = 1.
\]
With flat connection, one thus has holonomy representation of the fundamental group 
$\rho: \pi_1(S) \to \mathbb{C}^*$, where $\mathbb{C}^*= \mathbb{C}\setminus\{0\}$. Namely for any closed loop $\gamma \in \pi_1(S)$, one defines
\[
\rho(\gamma) = \prod_{uv \in \gamma} \phi_{uv}.
\]
The holonomy representation is then determined by the values on the generating set of loops $\gamma_1, \dots, \gamma_{2g}$ of $\pi_1(S)$ and we write $z_k:=\rho(\gamma_k)$ for $k=1,2,\dots,2g$. Such a flat $\mathbb{C}$-bundle is isomorphic to the flat $\mathbb{C}$-bundle where the connection $\phi$ is of the form
\[
\phi_{e} = \prod^{2g}_{k=1} z_k^{\iota(\gamma^i,e)}
\]
for every oriented edge $e$. When the connection is trivial—that is, when $z_1 = \dots = z_{2g} = 1$—we recover the usual untwisted setting.

Given a vector bundle over the vertices, we can further extend it to the edges of the graph. 
In addition to the vertex spaces, we assign a $1$-dimensional vector space $W_e \cong \mathbb{C}$ to each non-oriented edge $e$. 
For a vertex $u$ incident to an edge $e$, define connection isomorphisms $\phi_{ue} \colon W_u \to W_e$ and $\phi_{eu} \colon W_e \to W_u$ satisfying $\phi_{ue} = \phi_{eu}^{-1}$.
As before, we identify these connection isomorphisms with nonzero complex numbers acting by multiplication. These maps are chosen so that for an edge $e = uv$, the original connection on the vertex bundle satisfies
\[
\phi_{uv} = \phi_{ev} \cdot \phi_{ue}.
\] 

One convenient choice, which we adopt, is to fix an orientation for each edge: if $e = uv$ is oriented from $u$ to $v$, we set $\phi_{ue} := \phi_{uv}$ and $\phi_{ev} := 1$.

We define the \emph{connection incidence matrix} $d(z_1,\dots,z_{2g}) \in \mathbb{C}^{|E| \times |V|}$ as the linear operator that acts on functions $h : V \to \mathbb{C}$ via
\[
d h(uv) = \phi_{ve} h(v)- \phi_{ue} h(u),
\]
where $e=uv$ is an oriented edge from $u$ to $v$. With the above choice of $\phi_{ue}$ and $\phi_{ev}$, the matrix $d(z_1,\dots,z_{2g})$ has entries $1$, $-\phi_e$, or $0$, depending on whether a vertex is the head, the tail, or not incident to a given edge. 

We denote by $\mathbf{1}_n$ the $n$-dimensional vector with all entries equal to $1$, omitting the subscript when the dimension is clear.

At the trivial connection, we have $\phi_e = 1$ for all edges, and the operator $d(\mathbf{1})$ coincides with the incidence operator $d:\mathbb{R}^V \to \mathbb{R}^E$ defined in Section~\ref{sec:period}. With a mild abuse of notation, we therefore do not introduce a new symbol and continue to denote this operator by $d$. 
In particular, the constant vector $\mathbf{1}_V$ lies in the kernel of $d(\mathbf{1})$, namely
\begin{align}\label{for:d110}
d(\mathbf{1}) \cdot \mathbf{1}_V = 0.
\end{align} We further define $d^{\star}  \in \mathbb{C}^{|V| \times |E|}$ such that for every $1$-form $\omega$
\[
d^{\star}\omega(v)= \sum_{e=uv} \phi_{ev} \omega(e).
\]
As a matrix, one can verify that $d^{\star}=d(z_1^{-1}, \dots, z_{2g}^{-1})^t$.  

Let $\mathsf{c} : E \to \mathbb{R}_{>0}$ be positive edge weights, and let $\mathsf{C}$ be the $E \times E$ diagonal matrix representing $\mathsf{c}$. The \emph{bundle Laplacian} is defined by
\[
\Delta(z_1, \dots, z_{2g}) := d^{\star} \mathsf{C} d 
= d(z_1^{-1}, \dots, z_{2g}^{-1})^{t}\, \mathsf{C}\, d(z_1, \dots, z_{2g}),
\]
which acts on functions $h:V\to\mathbb{C}$ by
\[
(\Delta h)(v)
= \sum_{e=uv} \mathsf{c}_e \phi_{ev}\bigl(\phi_{ve}h(v)-\phi_{ue}h(u)\bigr)
= \sum_{e=uv} \mathsf{c}_e \bigl(h(v)-\phi_{uv}h(u)\bigr).
\]
In particular, $\Delta$ depends only on the vertex connection. 

At the trivial connection $z_k = 1$ for all $k$, $\Delta(\mathbf{1})$ coincides with the ordinary graph Laplacian $\Delta$ introduced in Section~\ref{sec:period}.

Define the polynomial 
\[
P(z_1, \dots, z_{2g}) := \det \Delta(z_1, \dots, z_{2g}).
\]

We first compute the first derivatives of $P$. For simplicity, we write $\boldsymbol{z} = (z_1, \dots, z_{2g})$ and $\boldsymbol{z}^{-1} = (z_1^{-1}, \dots, z_{2g}^{-1})$.

\begin{lemma}\label{lem:1d0}
For any $i \in \{1,\dots,2g\}$, the first derivative of $P$ at the trivial connection vanishes:
\begin{align}\label{for:1d0}
\left. \frac{\partial P}{\partial z_i} \right|_{\boldsymbol{z}=\mathbf{1}} = 0.
\end{align}
\end{lemma}

\begin{proof}
Introduce auxiliary variables $\boldsymbol{w} = (w_1,\dots,w_{2g})$ with $\boldsymbol{w} = \boldsymbol{z}^{-1}$. By the chain rule,
\[
\frac{\partial}{\partial z_i} P(\boldsymbol{z}^{-1})
= -\frac{1}{z_i^2} \left. \frac{\partial P}{\partial w_i} \right|_{\boldsymbol{w}=\boldsymbol{z}^{-1}}.
\]

By definition, the bundle Laplacian satisfies
$\Delta(\boldsymbol{z}^{-1}) = \Delta(\boldsymbol{z})^t$.
Therefore,
\[
P(\boldsymbol{z}^{-1}) = \det \Delta(\boldsymbol{z}^{-1}) = \det \Delta(\boldsymbol{z})^t = \det \Delta(\boldsymbol{z}) = P(\boldsymbol{z}).
\]
(That is, $P(\boldsymbol{z})$ is a Laurent polynomial in $\boldsymbol{z} = (z_1, \dots, z_{2g})$ and is \emph{reciprocal}~\cite{Kenyon2011}.)

Evaluating at the trivial connection $\boldsymbol{z} = \mathbf{1}$, we then have
\[
\left. \frac{\partial P}{\partial z_i} (\boldsymbol{z}) \right|_{\boldsymbol{z}=\mathbf{1}}
= \left. \frac{\partial P}{\partial z_i} (\boldsymbol{z}^{-1}) \right|_{\boldsymbol{z}=\mathbf{1}}
= - \left. \frac{\partial P}{\partial z_i} \right|_{\boldsymbol{z}=\mathbf{1}},
\]
which establishes \eqref{for:1d0}.
\end{proof}

We now turn to the computation of the second derivatives of $P$ at the trivial connection $\boldsymbol{z} = \mathbf{1}$. This will allow us to compute the Hessian explicitly, which is a crucial step towards the proof of Theorem~\ref{thm:Hess}.  

By Jacobi's formula, the first and second derivatives of $P$ can be expressed in terms of the bundle Laplacian as
\begin{align}
\frac{\partial P}{\partial z_i}(\boldsymbol{z}) &= \tr \Big( \operatorname{adj}(\Delta(\boldsymbol{z})) \cdot \frac{\partial \Delta}{\partial z_i} \Big), \notag \\
\frac{\partial^2 P}{\partial z_i \partial z_j}(\boldsymbol{z}) &= \tr \Big( \operatorname{adj}(\Delta(\boldsymbol{z})) \cdot \frac{\partial^2 \Delta}{\partial z_i \partial z_j} \Big)
+ \tr \Big( \frac{\partial \operatorname{adj}(\Delta(\boldsymbol{z}))}{\partial z_i} \cdot \frac{\partial \Delta}{\partial z_j} \Big).
\label{for:d2P}
\end{align}
Here, $\adj(\Delta(\boldsymbol{z}))$ denotes the adjugate of $\Delta(\boldsymbol{z})$, i.e., the transpose of its cofactor matrix, satisfying
\begin{align}\label{for:adjeq}
\Delta(\boldsymbol{z}) \cdot \adj(\Delta(\boldsymbol{z})) = \det(\Delta(\boldsymbol{z})) \cdot \mathbb{I}_V.
\end{align}

At the trivial connection, $\Delta(\mathbf{1})$ reduces to the usual graph Laplacian. By Kirchhoff's Matrix-Tree Theorem, its adjugate takes the simple form
\begin{align}\label{for:adjmtt}
\adj(\Delta(\mathbf{1})) = \det(\Delta_{\bar{o}\bar{o}}) \cdot \mathbf{1}_V \mathbf{1}_V^t,
\end{align}
where $\Delta_{\bar{o}\bar{o}}$ is the reduced Laplacian obtained by removing the row and column corresponding to a fixed vertex $o$, and $\mathbf{1}_V \mathbf{1}_V^t$ is the $V \times V$ matrix with all entries equal to $1$.

With these preparations, we can now compute each term in the Hessian explicitly. The following lemma carries out this computation, and we use its results immediately afterward to prove Theorem~\ref{thm:Hess}.

\begin{lemma}\label{lem:Hessian-general}
Let $d_i := \frac{\partial d}{\partial z_i}$. For any $i, j \in \{1, \dots, 2g\}$ and evaluating at $\boldsymbol{z} = \mathbf{1}$, we have
\begin{align}
\tr \left( \adj (\Delta(\boldsymbol{z})) \frac{\partial^2 \Delta}{\partial z_i \partial z_j} \right) &= -2 \det (\Delta_{\bar{o}\bar{o}}) (d_i \cdot \mathbf{1}_V)^t \mathsf{C} (d_j \cdot \mathbf{1}_V), \label{for:tr1}\\
\tr \left( \frac{\partial \adj (\Delta(\boldsymbol{z}))}{\partial z_i} \cdot \frac{\partial \Delta}{\partial z_j} \right) &= 2 \det (\Delta_{\bar{o}\bar{o}}) (d_i \cdot \mathbf{1}_V)^t \mathsf{C} d_{\bar{o}} \Delta_{\bar{o}\bar{o}}^{-1} d_{\bar{o}}^t \mathsf{C} (d_j \cdot \mathbf{1}_V).\label{for:tr2}
\end{align}
\end{lemma}

\begin{proof}
Let $d_{ij} := {\partial d}/{\partial z_i \partial z_j}$. Recall that $\Delta(\boldsymbol{z}) = d(\boldsymbol{z}^{-1})^t \, \mathsf{C} \, d(\boldsymbol{z})$. 

At $\boldsymbol{z} = \mathbf{1}$, it holds that
\[
\frac{\partial^2 \Delta}{\partial z_i \partial z_j} 
= 
\begin{cases}
d_{ij}^t \mathsf{C} d - d_j^t \mathsf{C} d_i - d_i^t \mathsf{C} d_j + d^t \mathsf{C} d_{ji} & \text{if } i \neq j, \\
2 d_i^t \mathsf{C} d + d_{ii}^t \mathsf{C} d - 2 d_i^t \mathsf{C} d_i + d^t \mathsf{C} d_{ii} & \text{if } i = j.
\end{cases}
\]

By~\eqref{for:d110},~\eqref{for:adjmtt}, and the cyclic property of the trace, $\operatorname{tr}(AB) = \operatorname{tr}(BA)$ (whenever $AB$ and $BA$ are defined), it follows that
\begin{align*}
\operatorname{tr} \left( \operatorname{adj}(\Delta(\boldsymbol{z})) \frac{\partial^2 \Delta}{\partial z_i \partial z_j} \right)
&= \det(\Delta_{\bar{o}\bar{o}}) \cdot \operatorname{tr} \left( \mathbf{1}_V \mathbf{1}_V^t \cdot \frac{\partial^2 \Delta}{\partial z_i \partial z_j} \right) \\
&= \det(\Delta_{\bar{o}\bar{o}}) \cdot \operatorname{tr} \left( \mathbf{1}_V^t \cdot \frac{\partial^2 \Delta}{\partial z_i \partial z_j} \cdot \mathbf{1}_V \right) \\
&= -2 \det(\Delta_{\bar{o}\bar{o}}) \cdot 
\left( (d_i \cdot \mathbf{1}_V)^t \mathsf{C} (d_j \cdot \mathbf{1}_V) \right)
\end{align*}
for all $i, j \in \{1, \dots, 2g\}$. This proves~(\ref{for:tr1}).

Let $\adj (\Delta)_i := {\partial \adj (\Delta)}/{\partial z_i}$. Next, applying~(\ref{for:d110}),~(\ref{for:1d0}) and~(\ref{for:adjmtt}), differentiating~(\ref{for:adjeq}) with respect to $z_i$ and evaluating at $\boldsymbol{z} = \mathbf{1}$ yields:
\begin{align*}
    d^t \mathsf{C} d \cdot \adj (\Delta)_i = \Delta \cdot \adj (\Delta)_i &= - \frac{\partial \Delta}{\partial z_i} \cdot \det(\Delta_{\bar{o}\bar{o}}) \cdot \mathbf{1}_V \mathbf{1}_V^t + \frac{\partial P}{\partial z_i} \cdot \mathbb{I}_V\\
    &= - \det(\Delta_{\bar{o}\bar{o}}) \cdot d^t \mathsf{C} d_i \cdot \mathbf{1}_V \mathbf{1}_V^t.
\end{align*}

We now decompose $d = (d_o \; d_{\bar{o}})$, where $d_o$ denotes the column indexed with the vertex $o$. Note that $d_o = -d_{\bar{o}} \cdot \mathbf{1}_{V-1}$ provided $\boldsymbol{z} = \mathbf{1}$. Removing the row of $o$ from the equation above gives:
\begin{align*}
    d_{\bar{o}}^t \mathsf{C} d \cdot \adj (\Delta)_i = - \det(\Delta_{\bar{o}\bar{o}}) \cdot d_{\bar{o}}^t \mathsf{C} d_i \cdot \mathbf{1}_V \mathbf{1}_V^t.
\end{align*}
Left-multiplying by $d_{\bar{o}} \Delta_{\bar{o}\bar{o}}^{-1}$:
\begin{equation}\label{for:dadj}
\begin{split}
    -\det(\Delta_{\bar{o}\bar{o}}) \, d_{\bar{o}} \Delta_{\bar{o}\bar{o}}^{-1} d_{\bar{o}}^t \mathsf{C} \, d_i \cdot \mathbf{1}_V \mathbf{1}_V^t 
    =\; & d_{\bar{o}} \Delta^{-1}_{\bar{o}\bar{o}} d_{\bar{o}}^t \mathsf{C} \, d \cdot \operatorname{adj} (\Delta)_i \\
    =\; & d_{\bar{o}}
    \begin{pmatrix}
        \vert & \vert  \\
        \Delta^{-1}_{\bar{o}\bar{o}} d_{\bar{o}}^t \mathsf{C} \, d_o & \Delta^{-1}_{\bar{o}\bar{o}} d_{\bar{o}}^t \mathsf{C} \, d_{\bar{o}} \\
        \vert & \vert
    \end{pmatrix} \operatorname{adj} (\Delta)_i \\
    =\; & d_{\bar{o}}  
    \begin{pmatrix}
        \vert & \vert  \\
        - \mathbf{1}_{V-1} & \mathbb{I}_{V-1} \\
        \vert & \vert
    \end{pmatrix} \operatorname{adj} (\Delta)_i \\
    =\; & d \cdot \operatorname{adj} (\Delta)_i.
\end{split}
\end{equation}

On the other hand, noting that $w_i = z_i^{-1}$ and ${\partial w_i}/{\partial z_i} = -z_i^{-2}$, we have:
\[
\left( \frac{\partial}{\partial z_i} \adj(\Delta(\boldsymbol{w})) \right)^t
= \frac{\partial}{\partial z_i} \adj(\Delta(\boldsymbol{z}))^t 
= \frac{\partial}{\partial z_i} \adj(\Delta(\boldsymbol{w})) 
= -\frac{1}{z_i^2} \cdot \frac{\partial \adj(\Delta)}{\partial w_i}.
\] This implies that at $\boldsymbol{z} = \mathbf{1}$, 
\begin{align}\label{for:T-}
    \left( \frac{\partial \adj(\Delta)}{\partial z_i} \right)^t = - \frac{\partial \adj(\Delta)}{\partial z_i}.
\end{align}

Thus, given $\boldsymbol{z} = \mathbf{1}$, we apply~(\ref{for:dadj}) and~(\ref{for:T-}) to conclude that
\begin{align*}
\tr \left( \frac{\partial \adj(\Delta)}{\partial z_i} \cdot \frac{\partial \Delta}{\partial z_j} \right) &= \tr\left( (\adj \Delta)_i \cdot (d^t \mathsf{C} d_j - d_j^t \mathsf{C} d) \right) \\
&= -2 \tr( d_j^t \mathsf{C} d \cdot (\adj \Delta)_i ) \\
&= 2 \det(\Delta_{\bar{o}\bar{o}}) \tr\left( d_j^t \mathsf{C} d_{\bar{o}} \Delta_{\bar{o}\bar{o}}^{-1} d_{\bar{o}}^t \mathsf{C} d_i \cdot \mathbf{1}_V \mathbf{1}_V^t \right) \\
&= 2 \det(\Delta_{\bar{o}\bar{o}}) (d_j \cdot \mathbf{1}_V)^t \mathsf{C} d_{\bar{o}} \Delta_{\bar{o}\bar{o}}^{-1} d_{\bar{o}}^t \mathsf{C} (d_i \cdot \mathbf{1}_V)\\
&= 2 \det(\Delta_{\bar{o}\bar{o}}) (d_i \cdot \mathbf{1}_V)^t \mathsf{C} d_{\bar{o}} \Delta_{\bar{o}\bar{o}}^{-1} d_{\bar{o}}^t \mathsf{C} (d_j \cdot \mathbf{1}_V).
\end{align*}
This proves (\ref{for:tr2}).
\end{proof}

Define the matrix
\[
M := \begin{pmatrix}
\vert & & \vert \\
d_1 \cdot \mathbf{1}_V & \cdots & d_{2g} \cdot \mathbf{1}_V \\
\vert & & \vert
\end{pmatrix} \in \mathbb{R}^{|E| \times 2g},
\]
whose columns represent discrete closed $1$-forms on the graph with prescribed periods corresponding to the generators of $H_1(S; \mathbb{Z})$. It is clear that the matrix $M$ coincides with the one defined in~(\ref{for:M}).

\begin{proof}[Proof of Theorem~\ref{thm:Hess}]
From~(\ref{for:d2P}) and Lemma~\ref{lem:Hessian-general}, we know that the Hessian of $P(\boldsymbol{z}) = \det \Delta(\boldsymbol{z})$ at the trivial connection $\boldsymbol{z} = \mathbf{1}$ is given by
\[
\left. \operatorname{Hess}(P) \right|_{z_1=z_2=\dots=z_{2g}=1} = 2 \det(\Delta_{\bar{o}\bar{o}}) \left( M^t \mathsf{C} d_{\bar{o}} \Delta_{\bar{o} \bar{o}}^{-1} d_{\bar{o}}^t \mathsf{C} M - M^t \mathsf{C} M \right).
\]
By Proposition~\ref{prop:Lop}, the right-hand side equals $-2 \det(\Delta_{\bar{o}\bar{o}}) \cdot \Omega L$. Hence,
\begin{align*}
     \left. \operatorname{Hess}(P)\right|_{z_1=z_2=\dots=z_{2g}=1} &= -2 \det(\Delta_{\bar{o} \bar{o}}) \cdot \Omega L. \qedhere
\end{align*}
\end{proof}

\begin{remark} \rm
    There is a known identity to express the polynomial $P(z_1,\dots,z_{2g})$ as a weighted sum over \emph{cycle-rooted spanning forests} (CRSF)~\cite{Forman1993,Kenyon2011}. Combining the identity with Theorem~\ref{thm:Hess} yields an alternative proof of Corollary~\ref{cor:k1} in terms of $1$-quasi-trees. Here our $1$-quasi-trees are cycle-rooted spanning forests whose cycles are not homologically trivial. One can verify that the terms involving other CRSF in the identity vanish when evaluating the second partial derivatives at $z_1=z_2=\dots=z_{2g}=1$. Similar calculation is also carried out by Kassel and Kenyon~\cite{Kassel2017}.
\end{remark}

\section{$g$-quasi-trees} \label{sec:gHQT}

In this section, we study $g$-quasi-trees on a surface of genus $g$ and their relation to the normalized period matrix. This provides a combinatorial perspective on the Weil--Petersson Kähler potential. We also give a topological interpretation of the determinants $\det \mathcal{T}_I$ in terms of connected components of certain lifted surfaces, which, however, is valid only in a restricted, special case.

\subsection{Normalized period matrix}\label{sec:normperiod}

We focus on the classical Riemann surface and show how the usual normalized period matrix is related to our approach. It provides a combinatorial interpretation of the potential of the Weil--Petersson Kähler metric on Teichmüller space in terms of $g$-quasi-trees.

Recall that we fixed a symplectic basis $\{\gamma_1,\dots,\gamma_{2g}\}$ for the homology of the surface. Consider a basis $\{\omega_1,\dots,\omega_g\}$ of the complex vector space of holomorphic $1$-forms and define $g \times g$ complex matrices $A$, $B$ via
\[
A_{ij}:= \int_{\gamma_i} \omega_j, \quad B_{ij}:= \int_{\gamma_{g+i}} \omega_j. 
\]
In the classical theory, the normalized period matrix is the $g \times g$ complex matrix
\[
\Pi:= A^{-1} B, 
\]
which is independent of the choice of the basis of holomorphic $1$-forms. It is known that $\Pi$ is symmetric and its imaginary part is positive definite.

In contrast, we define the period matrix $L$ mapping the periods of the real part of a holomorphic $1$-form to the periods of its imaginary part. Namely, the matrix $L$ is defined such that
\[
L \begin{pmatrix}
	\Re A &  \Im A \\ \Re B & \Im B
\end{pmatrix}
= \begin{pmatrix}
	\Im A &  -\Re A \\ \Im B & -\Re B 
\end{pmatrix}.
\]
We remark that in the classical setting, $L^2=-\mathbb{I}_{2g}$ and $\det L=1$, which is generally different from the discrete setting (cf.\ Corollary~\ref{cor:det}). Since the matrix $L$ is independent of the choice of the basis of holomorphic $1$-forms, without loss of generality, we can choose a basis of holomorphic $1$-forms such that the periods satisfy $A= \mathbb{I}_{g}$ and hence $B=\Pi$. In this case the matrix $ L$ satisfies
\[
L \begin{pmatrix}
	\mathbb{I}_g & 0 \\ \Re \Pi & \Im \Pi
\end{pmatrix}
= \begin{pmatrix}
	0 &   -\mathbb{I}_g \\ \Im \Pi & -\Re \Pi 
\end{pmatrix}.
\]
Recall that
\[
\Omega = \begin{pmatrix}
	0 & \mathbb{I}_g \\
	-\mathbb{I}_g & 0
\end{pmatrix}.
\]
Thus, we have 
\begin{align*}
	\Omega L =&  \begin{pmatrix}
		0 & \mathbb{I}_g \\
		-\mathbb{I}_g & 0
	\end{pmatrix} \cdot
\begin{pmatrix}
	0 &   -\mathbb{I}_g \\ \Im \Pi & -\Re \Pi 
\end{pmatrix}
	\cdot \begin{pmatrix}
		\mathbb{I}_g & 0 \\ \Re \Pi & \Im \Pi
	\end{pmatrix}^{-1} \\ 
	=& \begin{pmatrix}
		\Im \Pi + \Re \Pi (\Im  \Pi)^{-1} \Re \Pi & - \Re \Pi (\Im  \Pi)^{-1} \\ - (\Im  \Pi)^{-1} \Re \Pi & (\Im  \Pi)^{-1}
	\end{pmatrix}.
\end{align*}
Particularly, the K\"{a}hler potential can be written as
\[
 \log(\frac{\det'(\Delta)}{\det \Im \Pi}) =  \log( \det{ }'(\Delta) \det(\Omega L)_{I,J}),
\]
where $I=J=\{g+1,\dots,2g\}$.

Over a discrete Riemann surface, the analogous expression counts $g$-quasi-trees via Theorem~\ref{thm:HQT_count}:
\[
\log( \det\phantom{ }'(\Delta) \det(\Omega L)_{I,J}) = \log \sum_{g\text{-quasi-tree } T}  \left( (\det \mathcal{T}_{I})^2   \prod_{e \in T} \mathsf{c}_{e} \right).
\]

\subsection{Towards a combinatorial interpretation of $\det\mathcal{T}_I$} \label{sec:interpret}

A natural direction for further investigation is to interpret the $\det \mathcal{T}_I$ terms in the above formula in terms of combinatorial and/or topological invariants of homological quasi-trees, analogous to other versions of the Matrix-Tree Theorem with multiplicities~\cite{Zas1982, Forman1993, DKM_Trees}. Although a complete description in general remains open, we highlight a special, non-trivial case that suggests the possibility of such a formulation.

For the rest of this section, we maintain the standing assumption that $I=\{g+1,\ldots,2g\}$ and that $T$ is a $g$-quasi-tree which, when viewed as a ribbon graph, is embedded on a sphere (so the embedding has $g+1$ faces). Such objects serve as discrete analogues of Riemann spheres.

Construct an abelian covering space $\tilde{S}$ of $S$ whose Deck group is canonically isomorphic to the $g$-dimensional lattice $\Gamma := \mathbb{Z}^{\{\gamma_{g+1}, \ldots, \gamma_{2g}\}}$ as follows\footnote{An alternative construction that does not require choosing disjoint $\gamma^i$'s is to take the covering space corresponding to the normal subgroup of $\pi_1(S)$ generated by $\gamma_1, \ldots, \gamma_g$ together with the commutator subgroup of $\pi_1(S)$.}. Choose representatives $\gamma^i$ for $i \in I := \{g+1, \ldots, 2g\}$ such that they are pairwise disjoint, and cut along these curves, labeling the resulting boundary components by $\partial^i_+$ and $\partial^i_-$. For each $\boldsymbol{u} \in \Gamma$, take a copy $S_{\boldsymbol{u}}$ of the cut surface. Identify each curve $\gamma_i$ with a standard basis element $\boldsymbol{e}_i$ of $\Gamma$. Now, for each $\boldsymbol{u} \in \Gamma$ and each $i$, glue $\partial^i_-$ of $S_{\boldsymbol{u}}$ to $\partial^i_+$ of $S_{\boldsymbol{u}+\boldsymbol{e}_i}$ in an orientation-compatible way. This gluing is defined so that if a directed loop representing $\gamma_i$ begins at a point lifted to $S_{\boldsymbol{u}}$, then the endpoint of the loop lifts to $S_{\boldsymbol{u}+\boldsymbol{e}_i}$.

\begin{theorem} \label{thm:T_I_interpret}
    Let $\tilde{T}$ be the lift of $T$ into $\tilde{S}$. If $|\det \mathcal{T}_I| \neq 0$, then it is equal to the number of connected components of $\tilde{S} \setminus \tilde{T}$. Otherwise, if $\det \mathcal{T}_I = 0$, the complement $\tilde{S} \setminus \tilde{T}$ has infinitely many connected components.
\end{theorem}

The key ingredient in the proof of Theorem~\ref{thm:T_I_interpret} is the following lemma on the homology of $S\setminus T$. Some calculations carry through in more general situations, but the overall statement holds only in our case due to dimensional constraints. For simplicity, in the proofs of both Lemma~\ref{lem:T_I_interpret} and Theorem~\ref{thm:T_I_interpret}, we contract an arbitrary spanning tree of $T$ and assume $T$ has one vertex and $g$ edges $e_1,\ldots,e_g$ (not to be confused with the standard basis $\boldsymbol{e}_i$'s).

\begin{lemma}\label{lem:T_I_interpret}
    The pushforward maps $H_1(T;\mathbb{Z})\rightarrow H_1(S;\mathbb{Z})$ and $H_1(S\setminus T;\mathbb{Z})\rightarrow H_1(S;\mathbb{Z})$ are injective and have equal images.
\end{lemma}

\begin{proof}
    Take the open $\epsilon$-neighborhood $A$ of $T$ on the surface, and let $B$ be the complement of the (closed) $\epsilon/2$-neighborhood of $T$. Clearly, $A$ and $S \setminus T$ deformation retract to $T$ and $B$, respectively. So we may work with $H_1(A;\mathbb{Z})$ and $H_1(B;\mathbb{Z})$ instead. In fact, $B$ is homeomorphic to $S\setminus T$.

As a few simple topological observations: (1) $A$ is homotopy equivalent to a wedge of $g$ loops $e_1,\ldots,e_g$ and is connected; (2) $B$ is connected; (3) $A\cap B$ is homotopy equivalent to the boundary of $A$ (as well as of $B$), consisting of $g+1$ disjoint circles $f_1,\ldots, f_{g+1}$ corresponding to the boundaries of the $g+1$ faces in the intrinsic embedding of $T$; (4) when $f_r$'s are viewed as homology classes in $H_1(A;\mathbb{Z})$ (respectively, $H_1(B;\mathbb{Z})$) with a consistent orientation induced from $A$ (respectively, $B$), the only relation among them is $\sum_{r=1}^{g+1} f_r=0$ and their integer multiples.

Consider the reduced Mayer--Vietoris exact sequence of homology (we omit the coefficient $\mathbb{Z}$ for brevity):
    \begin{align*}
H_2(S)\xrightarrow{\partial_2}H_1(A\cap B)\xrightarrow{(i_*,j_*)}H_1(A)\oplus H_1(B)\xrightarrow{k_*-l_*}H_1(S)\\
\xrightarrow{\partial_1} \tilde{H}_0(A\cap B)\xrightarrow{(i_*,j_*)} \tilde{H}_0(A)\oplus\tilde{H}_0(B)=0.
    \end{align*}

We temporarily tensor the sequence with $\mathbb{Q}$ to get some numerical invariants. $\partial_2$ is injective (over either $\mathbb{Z}$ or $\mathbb{Q}$) as it sends the fundamental class of $S$ to $\sum_r f_r\in H_1(A\cap B)$. By a simple dimension count using the aforementioned observations, $\dim H_1(A;\mathbb{Q})\oplus H_1(B;\mathbb{Q})=2g$ and in particular $\dim H_1(B;\mathbb{Q})=2g-\dim H_1(A;\mathbb{Q})=g$. By the classification of compact, orientable surfaces (with boundary), the closure of $B$ having $g+1$ boundary components while having a $g$-dimensional $1$-homology implies it has genus zero.

Working with $\mathbb{Z}$-coefficients again. $k_*$ is injective as $T$ is a $g$-quasi-tree. $l_*$ is also injective: if $l_*(\gamma)=0$, then $(0,-\gamma)\in \ker(k_*-l_*)$ implies $\gamma=j_*(\beta)$ for some $\beta\in\ker i_*$, which must be a multiple of $\sum_{r=1}^{g+1} f_r$ as stated above, but this forces $\gamma=0$ as well.

Since $B$ is topologically a sphere with $g+1$ disjoint closed disks removed, we can include $g$ new edges to $T$ and form $T'$ so the complement $S\setminus T'$ is a disk, which makes (the pushforward of) $H_1(T';\mathbb{Z})$ equals $H_1(S;\mathbb{Z})$ as in the proof of Lemma~\ref{lem:TLC_homology}.
Recall that a sublattice $\Gamma'\subset\Gamma\cong\mathbb{Z}^r\subset\mathbb{Q}^r$ is {\em saturated} if $\Gamma'=(\Gamma'\otimes\mathbb{Q})\cap\Gamma$, which is equivalent to $\Gamma/\Gamma'$ being torsion-free. A lattice basis of $k_*(H_1(A;\mathbb{Z}))$ can thus be extended to a lattice basis of $H_1(S;\mathbb{Z})$, and $k_*(H_1(A;\mathbb{Z}))$ is a saturated sublattice.

On the other hand, in the intrinsic embedding of $T$, each $e_s$ separates the sphere into two connected components, and summing the boundaries $f_r$'s of the faces on any one side (with consistent orientations) yields $e_i$. Therefore when viewing $f_r$'s as linear combinations of $e_s$, these elements conversely span $e_s$'s over $\mathbb{Z}$. By the injectivity of $k_*,l_*$, we have $l_*(H_1(B;\mathbb{Z}))\supset k_*(H_1(A;\mathbb{Z}))$. Since the image of $\partial_1$ is a sublattice of $\tilde{H}_0(A\cap B;\mathbb{Z})$ hence torsion-free, the image $(k_*-l_*)[H_1(A;\mathbb{Z})\oplus H_1(B;\mathbb{Z})]=l_*(H_1(B;\mathbb{Z}))$ is a saturated sublattice of $H_1(S;\mathbb{Z})$ as it is equal to the kernel of $\partial_1$ by exactness.

Summarizing, $l_*(H_1(B;\mathbb{Z}))\supset k_*(H_1(A;\mathbb{Z}))$ are two saturated sublattices of $H_1(S;\mathbb{Z})$. Their equality follows from the equality of their ranks; thus, it suffices to show that $\dim k_*(H_1(A;\mathbb{Q}))=\dim l_*(H_1(B;\mathbb{Q}))$. Since $k_*(H_1(A;\mathbb{Z})), l_*(H_1(B;\mathbb{Z}))$ are saturated, their respective quotients are torsion-free. Hence $k_*,l_*$ remain injective after tensoring with $\mathbb{Q}$, and the desired equality of dimension follows from the identities $\dim H_1(B;\mathbb{Q})=g=\dim H_1(A;\mathbb{Q})$ verified above.
\end{proof}

\begin{proof}[Proof of Theorem~\ref{thm:T_I_interpret}]

Fix a point $x\in S\setminus T$ and let $x_{\boldsymbol{u}}$ denote its lift in the fundamental domain indexed by $\boldsymbol{u}$. Define an equivalence relation on $\Gamma$ by declaring $\boldsymbol{u}\sim \boldsymbol{v}$ if $x_{\boldsymbol{u}}$ and $x_{\boldsymbol{v}}$ lie in the same connected component of $\tilde{S}\setminus \tilde{T}$, i.e., if there exists a path $\hat{\gamma}$ in $\tilde{S}\setminus \tilde{T}$ connecting $x_{\boldsymbol{u}}$ and $x_{\boldsymbol{v}}$. The path $\hat{\gamma}$ projects to a loop $\gamma$ in $S\setminus T$.

By Lemma~\ref{lem:T_I_interpret}, the maps
$H_1(T;\mathbb{Z}) \to H_1(S;\mathbb{Z})$ and
$H_1(S\setminus T;\mathbb{Z}) \to H_1(S;\mathbb{Z})$
are injective with equal images. Hence $\gamma$ can be viewed as a class in
$H_1(S\setminus T;\mathbb{Z}) \cong H_1(T;\mathbb{Z})$, spanned by the homology classes represented by the $e_i$'s.

Moreover, when $\gamma$ is expressed as $\sum_{i=1}^{2g} c_i \gamma_i$, the difference $\boldsymbol{u}-\boldsymbol{v}$ is equal to the truncation $\sum_{i=g+1}^{2g} c_i \boldsymbol{e}_i$. In particular, $e_i$ is represented by the $i$-th column of $\mathcal{T}_I$.
    
    Since $H_1(T;\mathbb{Z})$ is spanned by $e_i$'s, $\boldsymbol{u}\sim\boldsymbol{v}$ if and only if $\boldsymbol{u}-\boldsymbol{v}$ is in the $\mathbb{Z}$-column span of $\mathcal{T}_I$. Hence, the number of connected components of $\tilde{S} \setminus \tilde{T}$ is equal to the index of the sublattice generated by the columns of $\mathcal{T}_I$ in $\Gamma$. This index is equal to $|\det \mathcal{T}_I|$ if it is nonzero, and is infinite otherwise.
\end{proof}

\section{Some combinatorial properties of quasi-trees} \label{sec:HQT_combin}

Different notions of quasi-trees have appeared in the literature, notably one in topological graph theory. We relate these notions in this section, with an emphasis on their matroidal aspect.

\subsection{Comparison with ribbon-graph quasi-trees} \label{sec:compare}

In topological graph theory, a (spanning) subgraph $T$ of a cellularly embedded graph $G$ is considered as a quasi-tree if the complement $S\setminus (T\cup {\overline{T}^{\star}})$ is connected~\cite{Bouchet1989}; equivalently, the $\epsilon$-neighborhood of $T$ on the surface only has one boundary component~\cite{CMNR2019}. We refer to such a subgraph as a {\em ribbon-graph quasi-tree}. 

We show that ribbon-graph quasi-trees are special cases of our homological quasi-trees, though the converse does not hold in general. This fact has been observed in various contexts by others, including Richter~\cite{Richter}, whose work is discussed in the next subsection; we include a short proof here for completeness.

\begin{corollary}
    Suppose a subgraph $T$ of a cellularly embedded graph $G$ satisfies that $S\setminus (T\cup {\overline{T}^{\star}})$ is connected. Then $T$ is a $k$-homological quasi-tree where $k$ is even.
\end{corollary}

\begin{proof}
    Let $T$ be a ribbon graph quasi-tree. Following the proof of the first half of Proposition~\ref{prop:HQT_duality}, if $(V^\star,{\overline{T}^{\star}})$ is not connected, then $T$ contains a $1$-cycle that encloses a proper set of faces and $S\setminus T$ would not have been connected. By the symmetric role of $T$ and ${\overline{T}^{\star}}$, $T$ is also connected. Therefore $T$ (respectively, ${\overline{T}^{\star}}$) contains a spanning tree of $G$ (respectively, $G^\star$) and we can apply the second half of Proposition~\ref{prop:HQT_duality}.
    The evenness of $k$ is known in~\cite[Theorem~5.3]{Bouchet1989}. 
\end{proof}

\begin{example} \rm
    Any loop that represents a non-trivial homology, viewed as an embedded cycle by inserting vertices, is a simple example of $1$-homological quasi-tree that is not a ribbon-graph quasi-tree.
    We illustrate an example of even rank here. On the left-hand side of Figure~\ref{fig:HvsRG}, a graph $G$ and its dual, both embedded on the double torus, are depicted. The subgraph of $G$ shown on the right-hand side is a $2$-homological quasi-tree, but it is not a ribbon-graph quasi-tree.
\end{example}

\begin{figure}[!ht]
	\centering
	\includegraphics[scale=.5]{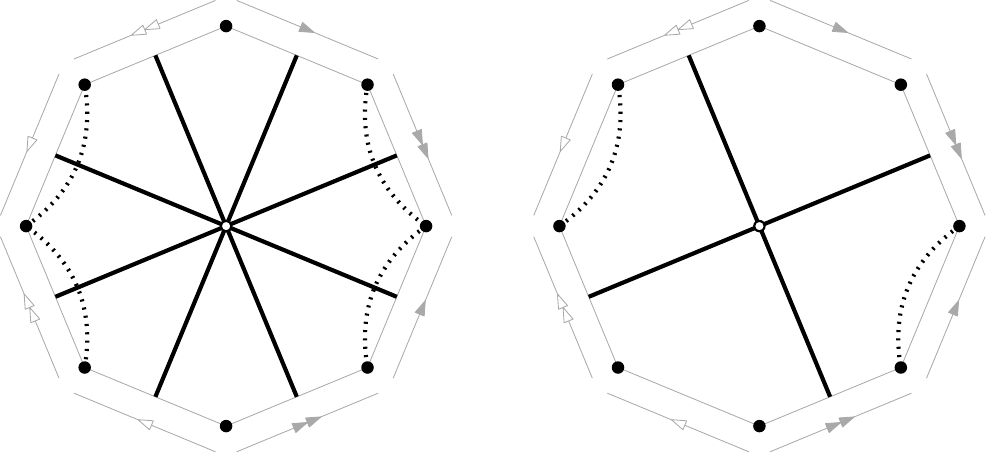}
\caption{On the left-hand side: a graph (represented by a white vertex and four solid edges) and its dual (represented by a black vertex and four dotted edges), both embedded on the double torus. On the right-hand side: a subgraph of the original graph.}
	\label{fig:HvsRG}
\end{figure}

When $k=2g$ is maximal, a $2g$-homological quasi-tree $T$ has $|V|-1+2g$ edges and the complement of $T$ is a disk. The set ${\overline{T}^{\star}}$ is a spanning tree on the dual graph and thus contractible (Lemma~\ref{lem:TLC_homology}). It implies that $S\setminus (T\cup {\overline{T}^{\star}})$ is connected. Therefore, a $2g$-homological quasi-tree is always a ribbon-graph quasi-tree.

\begin{corollary}
    Suppose $T$ is a $2g$-homological quasi-tree. Then it satisfies the condition that $S\setminus (T\cup {\overline{T}^{\star}})$ is connected.
\end{corollary}

\subsection{The (delta-)matroidal aspects of quasi-trees}

We first state the definition of matroids and delta-matroids~\cite{Bouchet1989}. Here $A\triangle B$ is the symmetric difference of two sets.

\begin{definition}
    A non-empty collection $\mathcal{F}$ of subsets of $E$ is a {\em delta-matroid} if for any $T,T'\in\mathcal{F}$ and $e\in T\triangle T'$, there exists $f\in T\triangle T'$ such that $T\triangle\{e,f\}\in\mathcal{F}$. It is a {\em matroid} if, in addition, whenever $e\in T\setminus T'$, $f$ can be chosen from $T'\setminus T$ and vice versa.
\end{definition}

\begin{proof} [Proof of Theorem~\ref{thm:delta_mat}]
    Recall that the collection $M(G)$ of spanning trees of $G$, as well as the collection $M^\star(G^\star)$ of complement of spanning trees of $G^\star$, are matroids on $E$~\cite[Example~1.6]{Eurhuh2020}. By Proposition~\ref{prop:HQT_duality}, $T\subset E$ is a homological quasi-tree if and only if there exists bases $T_0\in M(G)$ and $E\setminus C_0\in M^\star(G^\star)$ such that $T_0\subset T\subset E\setminus C_0$. Hence, the collection of homological quasi-trees form a {\em saturated} delta-matroid (a delta-matroid in which $F_1,F_3\in\mathcal{F}, F_1\subset F_2\subset F_3$ imply $F_2\in\mathcal{F}$). In particular, each level is a matroid~\cite[Section~2.2]{Eurhuh2020}.
\end{proof}

In fact, more can be said about homological quasi-trees once the matroidal connection has been established. Since they are tangential to the main topic of this article, we only highlight a few essential ones.

The map $M^\star(G^\star)\rightarrow M(G)$ induced by the identity map on ground set $E$ is a matroidal {\em quotient}, hence we can apply~\cite[Theorem~1.3]{Eurhuh2020} and obtain log-concavity results concerning the naive count of homological quasi-trees.

\begin{corollary}
    Let $h_{|V|-1+k}$ be the number of $k$-homological quasi-trees. Then $h_{|V|-1},\ldots,h_{|V|-1+2g}$ is a {\em strongly log-concave} sequence, i.e., $\frac{h_l^2}{\binom{|E|}{l}^2}\geq \frac{h_{l-1}}{\binom{|E|}{l-1}}\frac{h_{l+1}}{\binom{|E|}{l+1}}$ for $|V|-1<l<|V|-1+2g$.
\end{corollary}

In~\cite{Richter}, Richter considered, for each pair of positive integers $p,q$, the collection $\mathcal{H}(G;p,q)$ of subgraphs $T$ such that $T$ (respectively, ${\overline{T}^{\star}}$) has at most $p$ (respectively, $q$) connected components, and proved each $\mathcal{H}(G;p,q)$ is a {\em $g$-matroid} in the sense of Tardos~\cite{Tardos1985}. Again using Proposition~\ref{prop:HQT_duality}, we have:

\begin{corollary}
    The collection of homological quasi-trees is precisely $\mathcal{H}(G;1,1)$, and hence it is a $g$-matroid.
\end{corollary}

\section{Examples} \label{sec:example}

In this section, we illustrate the concept of homological quasi-trees with explicit examples on a surface of genus $2$. 

Consider the graph $G$ cellularly embedded on a genus-2 surface, obtained by identifying the sides of an octagon according to the usual genus-$2$ pattern, as shown in Figure~\ref{fig:genus2_graph}. It has 13 vertices in the interior of the octagon and 26 edges. The edges are labelled in the figure. The loops $\gamma^1,\gamma^2,\gamma^3,\gamma^4$ consist of edges from the dual cell decomposition and form a symplectic basis of the first homology of the genus-2 surface. With edge weights $\mathsf{c}_1,\mathsf{c}_2,\dots,\mathsf{c}_{26} \in \mathbb{R}_{>0}$ assigned to the edges of $G$, it defines a discrete Riemann surface together with the discrete period matrix $L$.

\begin{figure}[h]
    \centering
        \includegraphics[width=0.6\textwidth]{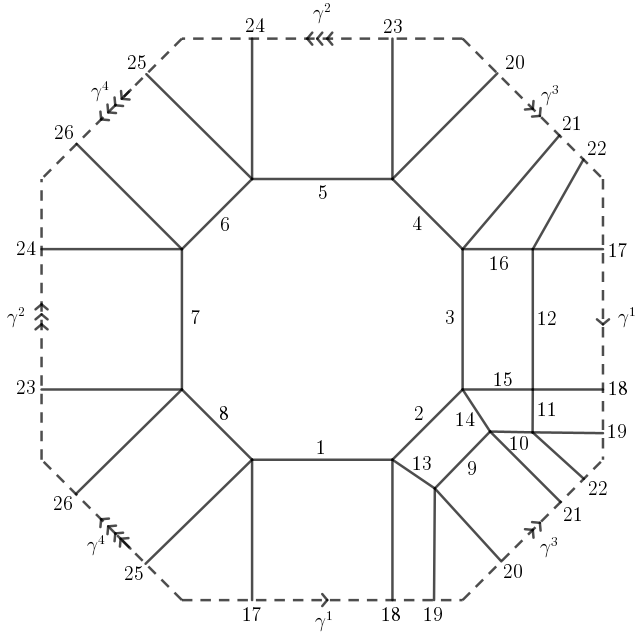}
    \caption{A graph $G$ cellularly embedded on a genus-$2$ surface.}
    \label{fig:genus2_graph}
\end{figure}

\subsection*{Example: A $1$-quasi-tree}

Figure~\ref{fig:genus2_1qt} displays a subgraph $T_1$ of $G$. Its edges are highlighted in bold. It is a spanning connected subgraph with $|V|=13$ edges and hence contains exactly one cycle. The cycle intersects the loop $\gamma^3$ twice transversely. If we take the index set $I=J=\{3\}$, the determinant of the associated intersection matrix
\[
\det \mathcal{T}_I = \det \mathcal{T}_J = \pm 2,
\]
where the sign depends on the orientation of the cycle and that of $\gamma^3$. However, the product $\det \mathcal{T}_{I} \cdot \det \mathcal{T}_{J}$ is always independent of the choice of the orientation.

In the summation with $I=J=\{3\}$,
\[
\sum_{\text{$1$-quasi-tree } T} \left( \det \mathcal{T}_{I} \cdot \det \mathcal{T}_{J} \cdot \prod_{e \in T} \mathsf{c}_e \right)
\]
the graph $T_1$ corresponds to the term
\[
4 \mathsf{c}_1 \mathsf{c}_{3}  \mathsf{c}_{4} \mathsf{c}_{5} \mathsf{c}_6 \mathsf{c}_7 \mathsf{c}_9 \mathsf{c}_{11} \mathsf{c}_{13} \mathsf{c}_{15} \mathsf{c}_{17} \mathsf{c}_{21} \mathsf{c}_{22}.
\]

\begin{figure}[h!]
    \centering
        \includegraphics[width=0.6\textwidth]{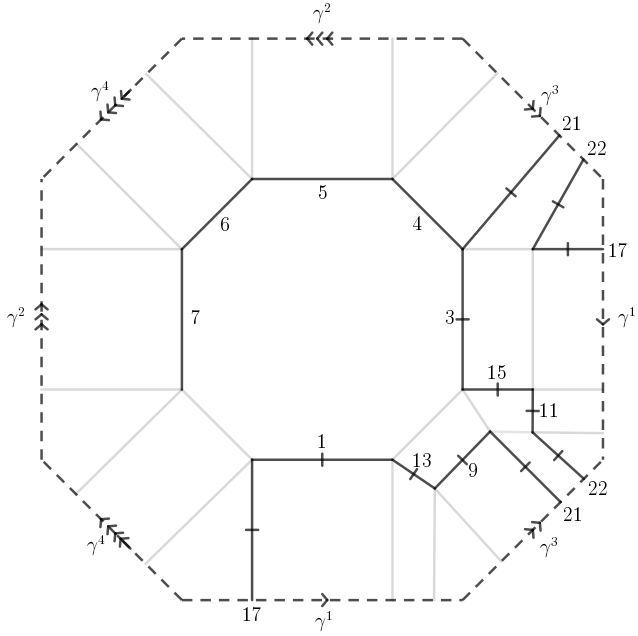}
    \caption{A $1$-quasi-tree $T_1$ (thick edges) on the genus-2 surface. It contains a unique cycle that is homologically non-trivial on the surface. The cycle consists of the edges marked with a transverse stroke.}
    \label{fig:genus2_1qt}
\end{figure}

\subsection*{Example: A $2$-quasi-tree}
Figure~\ref{fig:genus2_2qt} shows a subgraph $T_2$ which is the union of $T_1$ together with the edge labeled by $26$. The added edge yields a new cycle that intersects the loop $\gamma^4$ once transversely and has no intersection with the other loops in the basis. If we take the index set $I=J=\{3,4\}$, the associated intersection matrix $\mathcal{T}_I= \mathcal{T}_{J}$ is a 2 by 2 matrix and has determinant
\[
\det \mathcal{T}_I = \det \mathcal{T}_J = \pm 2.
\]
In the summation with $I=J=\{3,4\}$,
\[
\sum_{\text{$2$-quasi-tree } T} \left( \det \mathcal{T}_{I} \cdot \det \mathcal{T}_{J} \cdot \prod_{e \in T} \mathsf{c}_e \right)
\]
the graph $T_2$ corresponds to the term
\[
4 \mathsf{c}_1 \mathsf{c}_{3}  \mathsf{c}_{4} \mathsf{c}_{5} \mathsf{c}_6 \mathsf{c}_7 \mathsf{c}_9 \mathsf{c}_{11} \mathsf{c}_{13} \mathsf{c}_{15} \mathsf{c}_{17} \mathsf{c}_{21} \mathsf{c}_{22} \mathsf{c}_{26}.
\]

\begin{figure}[h!]
    \centering
        \includegraphics[width=0.6\textwidth]{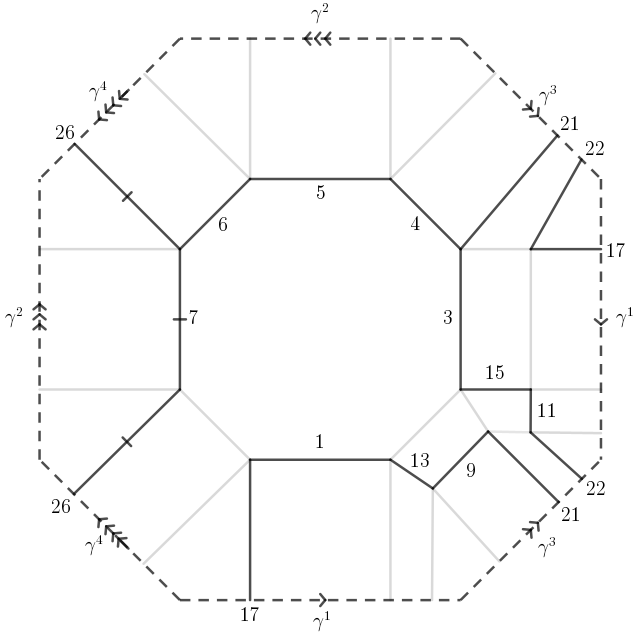}
    \caption{A $2$-quasi-tree $T_2$ (thick edges) on the genus-2 surface. It is the union of $T_1$ together with the edge labeled by $26$. The added edge forms a new cycle, consisting of the edges marked with a transverse stroke. This new cycle, together with the cycle from $T_1$, forms a basis of $H_1(T_2;\mathbb{Z})$.}
    \label{fig:genus2_2qt}
\end{figure}

\section*{Acknowledgements}
We would like to thank Cédric Boutillier for discussions leading to Theorem~\ref{thm:Hess}, Sang-il Oum for referring us to a paper by Richter, and Iain Moffatt and Donggyu Kim for pointing out references on delta-matroids.

\bibliographystyle{plainurl}
\bibliography{reference}
\end{document}